\documentclass{amsart}
\usepackage{amssymb, amsmath, amsthm, graphics, comment, xspace, enumerate, amscd}
\usepackage{color}
\usepackage{graphicx,amssymb,mathrsfs,amsmath}
\newtheorem{thm}{Theorem}[section]
\newtheorem{cor}[thm]{Corollary}
\newtheorem{lem}[thm]{Lemma}
\newtheorem{prop}[thm]{Proposition}
\theoremstyle{definition}
\newtheorem{defn}[thm]{Definition}
\newtheorem{example}[thm]{Example}
\theoremstyle{remark}
\newtheorem{rem}[thm]{Remark}
\numberwithin{equation}{section}

\begin{document}
\title[$({\mathrm R},{\mathcal B})$-Multidimensional-almost automorphic functions with applications]{$({\mathrm R},{\mathcal B})$-Multidimensional-almost automorphic functions with applications to integral and partial differential equations}

\author{A. Ch\'avez}
\address{Departamento de
Matem\'aticas, Facultad de Ciencias F\' isicas Y Matem\'aticas, Universidad Nacional de Trujillo, Trujillo, Per\'u}
\email{ajchavez@unitru.edu.pe}

\author{K. Khalil}
\address{Faculty of Sciences Semlalia, Cadi Ayyad University, B.P. 2390, 40000 Marrakesh, Morocco}
\email{kamal.khalil.00@gmail.com}

\author{M. Kosti\' c}
\address{Faculty of Technical Sciences,
University of Novi Sad,
Trg D. Obradovi\' ca 6, 21125 Novi Sad, Serbia}
\email{marco.s@verat.net}

\author{M. Pinto}
\address{Departamento de
Matem\'aticas, Facultad de Ciencias, Universidad de Chile, Santiago de Chile, Chile}
\email{pintoj.uchile@gmail.com}


{\renewcommand{\thefootnote}{} \footnote{2010 {\it Mathematics
Subject Classification.} 42A75, 43A60, 47D99.
\\ \text{  }  \ \    {\it Key words and phrases.} $({\mathrm R},{\mathcal B})$-Multidimensional-almost automorphic functions, asymptotically $({\mathrm R},{\mathcal B})$-multi-almost automorphic functions,
abstract Volterra integrodifferential equations.
\\  \text{  }  
Marko Kosti\' c is partially supported by grant 451-03-68/2020/14/200156 of Ministry
of Science and Technological Development, Republic of Serbia.
Manuel Pinto is partially supported by Fondecyt 1170466.}}

\begin{abstract}
In this paper, we analyze the classes of $({\mathrm R},{\mathcal B})$-multi-almost automorphic functions and asymptotically $({\mathrm R},{\mathcal B})$-multi-almost automorphic functions.
We provide plenty valuable applications to the abstract Volterra integrodifferential equations in Banach spaces and partial differential equations.
\end{abstract}
\maketitle
\section{Introduction and preliminaries}
S. Bochner introduced the concept of almost automorphy in the literature (\cite{Boch1,Boch2}) as a generalization of the almost periodicity. He discovered the class of almost automorphic functions when he was studying problems related to differential geometry (\cite{Boch3}), and also he used the almost automorphy of functions to prove some results about almost periodic functions. Afterward, Veech \cite{veech,veech-prim} and many other researchers deeply investigated this concept. In connection with the study of the asymptotic behavior of differential equations, the concept of almost automorphy has been intensively appeared. In fact, it is well known that some almost periodic systems do not carry necessarily almost periodic dynamics (\cite{Johnson,Ortega,Shen-Yi}), whereas, these systems may have bounded oscillating solutions, these oscillations belong to a boarder class than the class of almost periodic functions, namely almost automorphic functions.

It is well known that, (usually mild) solutions to nonautonomous evolution differential equations satisfy a concrete integral equation in which the integral kernels are expressed using two parameter evolution families $\{U(t,s)\}_{t\geq s \geq 0}$, see \cite{Pazy}. Thanks to the bi-almost automorhy of the evolution operator $\{U(t,s)\}_{t\geq s \geq 0}$ we may prove the existence of almost automorphic solutions, see \cite{chavez1,chavez2,chavez3,chenlin,43-xiao} for more details. 

For $X$ a Banach space, roughly speaking, a continuous function $f:\mathbb{R}\times \mathbb{R}\to X$  is bi-almost automorphic if it behaves like an almost automorphic function, but in which the sequences $\{s_n\} \in \mathbb{R}^2$ (translations) are considered in the set $\mathrm{R}=\{(w,w)\, :\, w \in \mathbb{R}\}$ and not in the full domain $\mathbb{R}^2$.

The notion of a (positively) bi-almost automorphic function was introduced by T. J. Xiao et al. in \cite[(2009)]{43-xiao}. Three years later, Z. Chen and W. Lin employed this notion in their investigation of nonautonomous stochastic evolution equations \cite{chenlin}, see also \cite{chavez1}-\cite{chavez2} and \cite[Appendix A.3]{diagana}, where the authors analyzed the notion of bi-almost automorphic sequences. Besides the above-mentioned papers, we would like to mention a recent research study \cite{chavez3} by A. Ch\'avez, M. Pinto and U. Zavaleta, where the authors systematically analyzed the notion of bi-almost automorphy in the study of abstract nonlinear integral equations that are simultaneously of
advanced and delayed type, as well as the research studies \cite{chang-zheng} by Y.-K. Chang, S. Zheng, \cite{hujin} by Z. Hu, Z. Jin, \cite{zxia} by Z. Xia and \cite{zxiawang} by Z. Xia, D. Wang.


Now we would like to mention some works in which bi-almost automorphic functions had appeared
\begin{itemize}
\item In \cite{43-xiao}, the authors introduced the notion of a continuous bi-almost automorphic function and obtained conditions for studying pseudo almost automorphic mild  solutions of the following equations in $\mathbb{R}$ :
\begin{eqnarray*}
x'(t)&=&A(t)x(t)+f(t,x(t))\\
x'(t)&=&A(t)x(t)+f(t,x(t-h))\\
x'(t)&=&A(t)x(t)+f(t,x\left( \alpha(t, x(t))\right)\, .
\end{eqnarray*}
\item In \cite{chenlin}, the authors introduced the notion of square-mean bi-almost automorphic functions for a stochastic processes and analyze the existence of square-mean almost automorphic solutions of the following non-autonomous linear stochastic evolution equation :
\begin{eqnarray*}
dx(t)=A(t)x(t)dt+f(t)dt+\gamma(t)dW(t),
\end{eqnarray*}
for $f,\gamma$ stochastic process and $W$ a two-sided standard one-dimensional Brownian motion.

\item The notion of a discrete bi-almost automorphic function can be found in \cite{chavez1,chavez2}, where the authors have been used it in the study of the non-autonomous difference equation that appears in the study of the following hybrid system of a differential equation in finite dimensional Banach spaces:
\begin{eqnarray*}
x'(t)=A(t)x(t)+B(t)x([t])+f(t,x(t),x([t])).
\end{eqnarray*}

\item We also mention that, in \cite{chavez3} the authors have used the notion of bi-almost automorphy and the one of $\lambda$-boundedness in order to study the following nonlinear abstract integral equations of advanced and delayed type:
\begin{eqnarray*}
y(t)&=&f(t,y(t),y(a_0(t)))+\int_{-\infty}^{t}C_1(t,s,y(s),y(a_1(s)))ds  \\ &&+ \int_{t}^{+\infty}C_2(t,s,y(s),y(a_2(s)))ds\; .
\end{eqnarray*}
\end{itemize}

Observing the previous works (and references cited therein), we emphasize that the notion of bi-almost automorphic function is crucial in the study of the almost automorphic dynamics in differential, integrodifferential and difference equations.

One other important notion that brings light to our present work is that of $\mathbb{Z}$-almost automorphic functions (see \cite{chavez1,chavez2}). This notion, roughly speaking, says that a function (not necessarily continuous) $f:\mathbb{R}\to X$  (defined on $\mathbb{R}$) is $\mathbb{Z}$-almost automorphic if it behaves like an almost automorphic function but only with integer sequences (a subset of $\mathbb{R}$) translations in the variable $t\in \mathbb{R}$, but not with sequences of the full domain $\mathbb{R}$ (see Example \ref{exam001} below).
 
The main objectives in this work are twofold: the first one is to develop the basic theory of $({\mathrm R},{\mathcal B})$-multi-almost automorphic functions, and the second one is provide concrete applications; in fact we apply the results to multi-dimensional integral equations of Volterra type, integrodifferential equations and partial differential equations such as the classical heat equation and Poisson equation. In the last application (of Poisson equation) we introduce a new generalization of Sibuya's result for almost periodic functions \cite{YSibuya} using the uniform continuity of the solution, our result is new even for the classical almost automorphic functions.\\

The organization and main ideas of this paper can be briefly described as follows. 
The main purpose of paper is to investigate $({\mathrm R},{\mathcal B})$-multi-almost automorphic functions and asymptotically $({\mathrm R},{\mathcal B})$-multi-almost automorphic functions.
We will not consider here the notion of a positively $({\mathrm R},{\mathcal B})$-multi-almost periodicity and their generalizations (\cite{43-xiao}).

We assume heneceforth
that $(X,\| \cdot \|)$, $(Y, \|\cdot\|_Y)$ and $(Z, \|\cdot\|_Z)$ are complex Banach space. By
$L(X,Y)$ we denote the Banach algebra of all bounded linear operators from $X$ into
$Y$ with $L(X,X)$ being denoted $L(X)$. If $A: D(A) \subseteq X \mapsto X$ is a closed linear operator,
then its nullspace (or kernel) and range will be denoted respectively by
$N(A)$ and $R(A)$. By $[D(A)]$ we denote the Banach space 
$(D(A), \|\cdot\|_{[D(A)]}),$ where $\|\cdot\|_{[D(A)]}$ is the graph norm defined by $\|x\|_{[D(A)]}:=\|x\|+\|Ax\|$ for all $x\in D(A).$

For given real numbers $s\in {\mathbb R}$ and $\theta \in (0,\pi]$, we define $\lceil s \rceil:=\inf \{
l\in {\mathbb Z} : s\leq l \}$ and $\Sigma_{\theta}:=\{ z\in {\mathbb C} \setminus \{0\} :
|\arg (z)|<\theta \}.$ The Euler Gamma function is denoted by
$\Gamma(\cdot)$. We also set $g_{\zeta}(t):=t^{\zeta-1}/\Gamma(\zeta),$ $\zeta>0.$
The convolution operator $\ast$ is defined by $f\ast g(t):=\int_{0}^{t}f(t-s)g(s)\,
ds.$ 

The symbol $C(I: X),$ where $I={\mathbb R}$ or $I=[0,\infty),$ stands for the space of all $X$-valued
continuous functions on the interval $I$. By $C_{b}(I: X)$ (respectively, $BUC(I: X)$) we denote the subspaces of $C(I: X)$ consisting of all bounded (respectively, all bounded uniformly continuous functions). Both $C_{b}(I: X)$ and $BUC(I: X)$ are Banach spaces with the sup-norm. This also holds for the space $C_{0}(I : X)$ consisting of all continuous functions $f : I \rightarrow X$ such that $\lim_{|t|\rightarrow +\infty}f(t)=0.$

Suppose that $f : {\mathbb R} \rightarrow X$ is continuous. Then it is said that \index{function!almost automorphic}
$f(\cdot)$ is almost automorphic if and only if for every real sequence $(b_{n})$ there exist a subsequence $(a_{n})$ of $(b_{n})$ and a map $g : {\mathbb R} \rightarrow X$ such that
\begin{align}\label{first-equ}
\lim_{n\rightarrow \infty}f\bigl( t+a_{n}\bigr)=g(t)\ \mbox{ and } \  \lim_{n\rightarrow \infty}g\bigl( t-a_{n}\bigr)=f(t),
\end{align}
pointwise for $t\in {\mathbb R}.$ If this is the case, then $f\in C_{b}({\mathbb R} : X)$ and the limit function $g(\cdot)$ is bounded on ${\mathbb R}$ but not necessarily continuous on ${\mathbb R}.$ Furthermore, if the convergence of limits appearing in \eqref{first-equ} is uniform on compact subsets of ${\mathbb R},$ then it is said that $f(\cdot)$ is compactly almost automorphic. Recall that an almost automorphic function $f(\cdot)$ is compactly almost automorphic if and only if it is uniformly continuous \cite[Lemma 3.7]{Es-sebbar}. 
Concerning two-parameter almost automorphic functions, we want only to recall that the authors of \cite{diagana}\ and \cite{nova-mono} have used the following notion:
A jointly continuous function $F : {\mathbb R} \times X \rightarrow Y$ is said to be almost
automorphic if and only if
for every sequence of real numbers $(s_{n}')$ there exists a subsequence $(s_{n})$ such
that
$$
G(t;x) := \lim_{n\rightarrow \infty}F\bigl(t +s_{n};x\bigr)
$$
is well defined for each $t \in {\mathbb R}$ and $x\in X,$ and
$$
\lim_{n\rightarrow \infty}G\bigl(t -s_{n} ; x\bigr) = F(t ; x)
$$
for each $t \in {\mathbb R}$ and $x\in X.$ This definition can be modified in the following sense which will be important in our consideration of case [L2] below (a general notion with a non-empty collection of subsets ${\mathcal B}$ of certain subsets of $X$ can be also introduced): 
A jointly continuous function $F : {\mathbb R} \times X \rightarrow Y$ is said to be almost
automorphic, uniformly on bounded subsets of $X$ if and only if for every bounded subset $B$ of $X$ and
for every sequence of real numbers $(s_{n}'),$ there exists a subsequence $(s_{n})$ such
that
$$
G(t;x) := \lim_{n\rightarrow \infty}F\bigl(t +s_{n};x\bigr)
$$
is well defined for each $t \in {\mathbb R}$ and $x\in X,$ uniformly for $x\in B$ and
$$
\lim_{n\rightarrow \infty}G\bigl(t -s_{n} ; x\bigr) = F(t ; x)
$$
for each $t \in {\mathbb R}$ and $x\in X,$ uniformly for $x\in B.$ \\

If the above limits converge uniformly on compact intervals of $ \mathbb{R}$, then we say that $F$ is compactly almost automorphic uniformly on bounded sets of $X$. 

\section{$({\mathrm R},{\mathcal B})$-Multi-almost automorphic type functions}\label{maremare}

Throughout this paper, we assume that $n\in {\mathbb N},$   
${\mathcal B}$ is a non-empty collection of subsets of $X$ and  ${\mathrm R}$ is a non-empty collection of sequences in ${\mathbb R}^{n};$  usually, ${\mathcal B}$ denotes the collection of all bounded subsets of $X$ or all compact subsets
of $X.$ Henceforth we will always assume that there exist $x\in X$ and $B\in {\mathcal B}$ such that $x\in B.$

In this section, we investigate $({\mathrm R},{\mathcal B})$-multi-almost automorphic functions and asymptotically $({\mathrm R},{\mathcal B})$-multi-almost automorphic functions. 

\begin{defn}\label{eovako}\index{function!(compactly) $({\mathrm R},{\mathcal B})$-multi-almost automorphic}
Suppose that $F : {\mathbb R}^{n} \times X \rightarrow Y$ is a continuous function. Then 
we say that the function $F(\cdot;\cdot)$ is $({\mathrm R},{\mathcal B})$-multi-almost automorphic if and only if for every $B\in {\mathcal B}$ and for every sequence $({\bf b}_{k}=(b_{k}^{1},b_{k}^{2},\cdot \cdot\cdot ,b_{k}^{n})) \in {\mathrm R}$ there exist a subsequence $({\bf b}_{k_{l}}=(b_{k_{l}}^{1},b_{k_{l}}^{2},\cdot \cdot\cdot , b_{k_{l}}^{n}))$ of $({\bf b}_{k})$ and a function
$F^{\ast} : {\mathbb R}^{n} \times X \rightarrow Y$ such that
\begin{align}\label{love12345678}
\lim_{l\rightarrow +\infty}F\bigl({\bf t} +(b_{k_{l}}^{1},\cdot \cdot\cdot, b_{k_{l}}^{n});x\bigr)=F^{\ast}({\bf t};x) 
\end{align}
and
\begin{align}\label{love123456789}
\lim_{l\rightarrow +\infty}F^{\ast}\bigl({\bf t} -(b_{k_{l}}^{1},\cdot \cdot\cdot, b_{k_{l}}^{n});x\bigr)=F({\bf t};x),
\end{align}
pointwise for all $x\in B$ and ${\bf t}\in {\mathbb R}^{n}.$ If the above limits converge uniformly on compact subsets of ${\mathbb R}^{n}$, then we say that 
$F(\cdot ; \cdot)$ is compactly $({\mathrm R},{\mathcal B})$-multi-almost automorphic.  By $AA_{({\mathrm R},{\mathcal B})}({\mathbb R}^{n} \times X : Y)$ and $AA_{({\mathrm R},{\mathcal B},{\bf c})}({\mathbb R}^{n} \times X : Y)$ we denote the spaces consisting of all 
$({\mathrm R},{\mathcal B})$-multi-almost automorphic functions and compactly $({\mathrm R},{\mathcal B})$-multi-almost automorphic functions, respectively.\index{space!$AA_{({\mathrm R},{\mathcal B})}({\mathbb R}^{n} \times X : Y)$} \index{space!$AA_{({\mathrm R},{\mathcal B},{\bf c})}({\mathbb R}^{n} \times X : Y)$}
\end{defn}
Notice that, in particular, if the function $F({\bf t};x) =F( {\bf t} )$ (i.e., $F : {\mathbb R}^{n} \rightarrow Y$), then we say that $F$ is ${\mathrm R}$-multi-almost automorphic and that space is denoted by $AA_{{\mathrm R}}({\mathbb R}^{n} : Y)$. 
\begin{rem} \label{Remark bi-aa}
The following special cases are very important:
\begin{itemize}
\item[L1.] Let ${\mathrm R}:=\{b : {\mathbb N} \rightarrow {\mathbb R}^{n} \, ; \, \mbox{ for all }j\in {\mathbb N}\mbox{ we have }b_{j}\in \{ (a,a,a,\cdot \cdot \cdot, a) \in {\mathbb R}^{n} : a\in {\mathbb R}\}\}.$  In the case where $n=2$ and ${\mathcal B}$ is the collection of all bounded subsets of $X,$ we say that the function $F(\cdot;\cdot)$ is bi-almost automorphic. The concept of bi-almost automorphy was introduced in literature in \cite[Definition 2.7]{chavez3} to study the asymptotic behavior of nonautonomous evolution equations in Banach spaces, see also \cite{Moi2,chavez1} and references therein. The research study \cite{chavez1} by A. Ch\'avez, S. Castillo and M. Pinto where the authors have used the notion of bi-almost automorphy of the Green functions in their investigation of almost automorphic solutions of abstract differential equations with piecewise constant arguments. Here, the pivot space is denoted in general by $X$ which equals to the pivot product space in \cite{chavez1}. Furthermore, in \cite{Moi2}, the authors proved the existence and uniqueness of $\mu$-pseudo almost automorphic solutions to a class of nonautonomous evolution equations with inhomogeneous boundary conditions, using the notion of bi-almost automorphic Green functions. In addition, the authors established sufficient weak conditions on the initial data of the equation insuring the bi-almost automorphy of the associated Green function, see \cite{Moi2}. In general case with the use of this collection of sequences in ${\mathbb R}^{n}$, we have that the function $F(\cdot ; \cdot)$ is $({\mathrm R},{\mathcal B})$-multi-almost automorphic if and only if for every $B\in {\mathcal B}$ and for every real sequence $(b_{k})$ there exist a subsequence $(a_{k})$ of $(b_{k})$ and a function
 $F^{\ast} : {\mathbb R}^{n} \times X \rightarrow Y$ such that
\begin{align*}
\lim_{k\rightarrow +\infty}\Bigl\| F\bigl({\bf t}
+(a_{k},\cdot \cdot\cdot, a_{k});x\bigr)-F^{\ast}({\bf t};x) \Bigr\|_{Y}=0
\end{align*}
and
\begin{align*}
\lim_{k\rightarrow +\infty}\Bigl\| F^{\ast}\bigl({\bf t}-(a_{k},\cdot \cdot\cdot, a_{k});x\bigr)-F({\bf t};x) \Bigr\|_{Y}=0,
\end{align*}
pointwise for all $x\in B$ and ${\bf t}\in {\mathbb R}^{n}.$
\item[L2.] 
${\mathrm R}=\{b : {\mathbb N} \rightarrow {\mathbb R}^{n} \, ; \, \mbox{ for all }j\in {\mathbb N}\mbox{ we have }b_{j}\in\{(a,0,0,\cdot \cdot \cdot, 0) \in {\mathbb R}^{n} : a\in {\mathbb R}\}\}.$ Without going into full details, we want only to consider here the case in which $X\in {\mathcal B}$ (the choice in which ${\mathcal B}$ is a collection of all bounded or compact subsets of $X$ is a bit complicated but the obtained conclusions are similar; the difficulty actually lies in the fact that a bounded (compact) set in the space ${\mathbb R}^{n-1} \times X$ is not necessarily a direct product of a bounded (compact) set in ${\mathbb R}^{n-1}$ and a bounded (compact) set in $X$); then it can be simply approved that 
the function 
$F(\cdot ;\cdot)$ is (compactly) $({\mathrm R},{\mathcal B})$-multi-almost automorphic if and only if the function ${\mathbb F} : {\mathbb R} \times {\mathcal X}
\rightarrow Y,$ given by 
${\mathcal F}(t ;{\mathcal x}):=F((t,{\bf t'}); x),$ $t\in {\mathbb R},$ ${\mathcal x} =({\bf t'}; x)\in {\mathcal X}={\mathbb R}^{n-1} \times X.$  
This implies that the notion introduced in 
Definition \ref{multi33} cannot be viewed as some special case of the notion of almost automorphic function from ${\mathbb R} \times X$ into $Y.$
\item[L3.] ${\mathrm R}$ is a collection of all sequences $b(\cdot)$ in ${\mathbb R}^{n}.$ This is the limit case in our analysis because, in this case, any 
$({\mathrm R},{\mathcal B})$-multi-almost automorphic function is automatically $({\mathrm R}_{1},{\mathcal B})$-multi-almost automorphic for any other collection ${\mathcal R}_{1}$ of sequences $b(\cdot)$ in ${\mathbb R}^{n}.$
\end{itemize}
\end{rem} 

In what follows we provide some examples illustrating the concept of $({\mathrm R},{\mathcal B})$-multi-almost automorphic functions. 
\begin{example}\label{exa01} Let the scalar function $ \varphi : \mathbb{R}\longrightarrow \mathbb{R} $ be almost automorphic and let $ (T(t))_{t\geq 0} \subset L(X)$ be uniformly bounded (i.e. $\sup_{t\in \mathbb{R}}\|T(t)\|_{L(X)} \leq M$) family of strongly continuous operators on $X$. Define a new function $G:\mathbb{R}^2\times X \to X$ by
$$G(t,s;x)=e^{\int_s^t \varphi(\tau) d\tau } T(t-s)x, \quad (t,s) \in \mathbb{R}^{2}, x \in X .$$ 
The function $G$ yields a bi-almost automorphic function in the sense of Remark \ref{Remark bi-aa}. For a detailed proof, we refer to  \cite[Example 7.1]{chenlin} (see also \cite[Example 4.1]{43-xiao}) where, in this case, $ (T(t))_{t\geq 0} $ presents an exponentially stable strongly continuous semigroup of bounded linear operators (\cite{Nag}). More generally, we have

\begin{enumerate}
\item  For $j\in \{1,2, \cdots, n-1\}$ let $\varphi_j:\mathbb{R}\to \mathbb{R}$ be a family of almost automorphic functions and $T_j$ a family of  semigroups that decays exponentially to $0$ as $t\to +\infty$. Then, $F: \mathbb{R}^n\times X\to X$ defined as follows
$$F(t_1,t_2,\cdots, t_n;x):=\sum_{j=1}^{n-1} T_j(t_{j+1}-t_{j})e^{\int_{t_{j}}^{t_{j+1}} \varphi_j(\xi)d\xi } x\, ,$$
is $({\mathrm R},{\mathcal B})$-multi-almost automorphic, in which $\mathrm{R}=\{b : {\mathbb N} \rightarrow {\mathbb R}^{n} \, ; \, \mbox{ for all }j\in {\mathbb N}\mbox{ we have }b_{j}\in \{ (a,a,a,\cdot \cdot \cdot, a) \in {\mathbb R}^{n} : a\in {\mathbb R}\}\},$  and ${\mathcal B}$ denotes the collection of all bounded subsets of $X.$

\item For $j\in \{1,2, \cdots, n\}$ let $\varphi_j:\mathbb{R}\to \mathbb{R}$ be a family of almost automorphic functions and $T_j$ a family of  semigroups that decays exponentially to $0$ as $t\to +\infty$. Thus,  $F: \mathbb{R}^{2n}\times X\to X$ defined as follows
$$F(t_1,t_2,\cdots, t_n;x):=\sum_{j=1}^{n} T_j(t_{2j}-t_{2j-1})e^{\int_{t_{2j-2}}^{t_{2j}} f_j(\xi)d\xi } x\, ,$$
is $({\mathrm R},{\mathcal B})$-multi-almost automorphic, in which $\mathrm{R}=\{b : {\mathbb N} \rightarrow {\mathbb R}^{n} \, ; \, \mbox{ for all }j\in {\mathbb N}\mbox{ we have }b_{j}\in \{ (a_1,a_1,a_2,a_2,\cdot \cdot \cdot, a_n,a_n) \in {\mathbb R}^{2n} : a_i\in {\mathbb R}\}\},$  and ${\mathcal B}$ denotes the collection of all bounded subsets of $X.$
\end{enumerate}
\end{example}

\begin{example}
The proof of this example shows haw to go in the proof of the next example. For $j=1,2, \cdots , n$, let $f_j$ be a family of almost automorphic function from $\mathbb{R}$ to $\mathbb{R}$. The function
$F:\mathbb{R}^{2n} \to \mathbb{R}$ defined by 
$$(s_1,s_2, \cdots , s_n,t_1,t_2,\cdots , t_n)\to F(s_1,s_2, \cdots , s_n,t_1,t_2,\cdots , t_n)=\prod_{j=1}^n \int _{s_j}^{t_j}f_j(\xi) d\xi\, ,$$
is $\mathrm{R}$-multi-almost automorphic function, where $\mathrm{R}=\{b : {\mathbb N} \rightarrow {\mathbb R}^{n} \, ; \, \mbox{ for all }j\in {\mathbb N}\mbox{ we have }b_{j}\in \{ (a_1,a_2,\cdots ,a_n,a_1,a_2,\cdots ,a_n) \in {\mathbb R}^{2n} : a_i\in {\mathbb R}\}\}$.
\end{example}

\begin{example}\label{ExMulti}
This example generalizes the previous one to multidimensional almost automorphic functions. Let $\mathrm{R}\subset \mathbb{R}^n$ be the set of sequences such that all its subsequences are also in $\mathrm{R}$ and  $\mathrm{R} '\subset \mathbb{R}^m$ be the set of sequences such that all its subsequences are also in $\mathrm{R}'$. For $i=1,2,\cdots ,p$ let $f_i:\mathbb{R}^n \to \mathbb{R}$ be $\mathrm{R}$-almost automorphic and for $j=1,2,\cdots ,q$ let $g_j:\mathbb{R}^m \to \mathbb{R}$ be $\mathrm{R}'$-almost automorphic. Define the new functions $F:\mathbb{R}^n \to \mathbb{R}^q$ by $F({\bf t})=\sum_{i=1}^p f_i({\bf t})e_i$ and $G:\mathbb{R}^m \to \mathbb{R}^q$ by $G({\bf s})=\sum_{j=1}^q g_j({\bf s})e_j$, where $\{e_i,\, i=1, \cdots , p\}$ is the canonical basis of $\mathbb{R}^p$ and $\{e_j,\, j=1, \cdots ,q\}$ is the canonical basis of $\mathbb{R}^q$. Now define the function $F\bigotimes G:\mathbb{R}^{n}\times \mathbb{R}^m\to M_{p\times q}(\mathbb{R})$ by

\begin{equation}
F\bigotimes G\, ({\bf t},{\bf s})= 
\begin{pmatrix}
f_1({\bf t})g_1({\bf s}) & f_1({\bf t})g_2({\bf s}) & \cdots & f_1({\bf t})g_q({\bf s})\\
f_2({\bf t})g_1({\bf s}) & f_2({\bf t})g_2({\bf s}) & \cdots & f_2({\bf t})g_q({\bf s})\\
\vdots & \vdots & \ddots & \vdots\\
f_p({\bf t})g_1({\bf s}) & f_p({\bf t})g_2({\bf s}) & \cdots & f_p({\bf t})g_q({\bf s})
\end{pmatrix}
\end{equation}
It is not difficult to see that $F\bigotimes G$ is $\mathrm{R}\times \mathrm{R}'$-almost automorphic, where $\mathrm{R}\times \mathrm{R}'=\{({\bf b},{\bf b'})\, : \, {\bf b} \in \mathrm{R}\, ,\, \,{\bf b'} \in \mathrm{R}'\}$. 

Observe that, for $p=q$ the function $F \widehat{\bigotimes} G:\mathbb{R}^{n}\times \mathbb{R}^m\times \mathbb{R}^p \to \mathbb{R}^p$ defined by 
$$F \widehat{\bigotimes} G\, ({\bf t},{\bf s},x)= F\bigotimes G\, ({\bf t},{\bf s})\cdot x\, \,$$
is $(\mathrm{R}\times \mathrm{R}',\mathcal{B})$-almost automorphic with $\mathcal{B}$ being the collection of bounded subsets of $\mathbb{R}^p$.
\end{example}

\begin{example}\label{exam001} In order to have more examples, we need to modify the continuity of $F$ in Definition \ref{eovako} accordingly. This example comes from the theory of differential equations with piecewise constant argument \cite{chavez1,chavez2}. In the mentioned papers, we found the following definition \cite[Definition 2.3]{chavez2} :
\begin{defn}
Let $f:\mathbb{R}\to \mathbb{C}^p$  be a bounded and piecewise continuous function. $f$ is said to be $\mathbb{Z}$-almost automorphic if for any sequence of integer numbers $\{s_n\}$, there exist a subsequence $\{s^{'}_n\} \subseteq \{s_n\}$ and a function $\tilde{f}:\mathbb{R}\to \mathbb{C}^p$ such that the following pointwise limits hold:
$$\lim_{n\to \infty}f(t+s_n^{'})=\tilde{f}(t)\, ,$$
$$\lim_{n\to \infty}f(t+s_n^{'})=\tilde{f}(t)\, .$$
\end{defn}
Let $A(t)$ be an almost automorphic matrix, and $\phi (t)$ a fundamental matrix solution of the system $x'(t)=A(t)x(t)$, then the Green matrix
$$ \Phi(t,s)=\phi(t)\phi^{-1}(s)\, ,$$
is Bi-almost automorphic \cite[Lemma 3.2]{chavez2}. Moreover the functions
$$\Phi(t,[t]),\, \, \, \, \int_{[t]}^t \Phi(t,\eta)B(\eta)d \eta\, ,$$
are $\mathbb{Z}$-almost automorphic, for $B$ an almost automorphc function, see \cite[Lemma 3.3]{chavez2}. Also,  if $g\in AA(\mathbb{R}\times \mathbb{C}^p \times \mathbb{C}^p; \mathbb{C}^p)$ and is uniformly continuous on compact subsets of $\mathbb{C}^p \times \mathbb{C}^p$ and $\psi \in AA(\mathbb{R}; \mathbb{C}^p)$, then the composition map $g(t,\psi(t), \psi ([t]))$ is $\mathbb{Z}$-almost automorphic. Whit this preliminary exposition we construct the piecewise continuous  function: $F:\mathbb{R}^3\times \mathbb{C}^p\to \mathbb{C}^p$, such that
\begin{eqnarray*}
F(t,s,u,\zeta)&=&||\zeta||_{\mathbb{C}^p}\left( \int_{[t]}^t\Phi(t, \eta) B(\eta) d \eta+ g(s,\psi(s),\psi([s]))\right.  \\ &&\left. +\sum_{j=1}^p\sin(\dfrac{1}{2+\cos(u) + \cos(\sqrt{2}u)})e_j \right),
\end{eqnarray*}
where $\{e_j,\, j=1, \cdots , p\}$ is the canonical basis of $\mathbb{R}^p$. Then, $F$ is $({\mathrm R},{\mathcal B})$-multi-almost automorphic, with $\mathrm{R}=\mathbb{Z}\times \mathbb{Z}\times \mathbb{R}$ and $\mathcal{ B}$ is the collection of bounded subsets of $\mathbb{C}^p$\, . 
\end{example}
\begin{rem}\label{ludar}
The first systematic study of almost automorphic functions on topological groups was conducted by 
W. A. Veech in \cite{veech}-\cite{veech-prim} (see also the papers \cite{Reich} by A. Reich and \cite{terras} by R. Terras). Following P. Milnes \cite{Milnes}, who considered only the scalar-valued case, we say that a continuous function $f : G\rightarrow  Y,$ where $G$ is a semitopological group, is almost automorphic if and only if for any sequence $(n_{i}')$ in $G$ there exists a subsequence $(n_{i})$ of $(n_{i}')$ such that the joint limit $\lim_{i,j}f(n_{i}n_{j}^{-1}t)=f(t)$ exists for all $t\in G.$ It is clear that ${\mathbb R}^{n} \times X$ is a semitopological group as well as that the notion introduced in Definition \ref{eovako} can be extended in this rather general framework. Observing that the element $x$ in the equations \eqref{love12345678}-\eqref{love123456789} is fixed, we have that any almost automorphic function $F : {\mathbb R}^{n} \times X \rightarrow Y$ in the sense of the above definition is $({\mathrm R},{\mathcal B})$-multi-almost automorphic, where ${\mathrm R}$ is a collection of all sequences $b(\cdot)$ in ${\mathbb R}^{n}$ and ${\mathcal B}$ is the collection of all subsets of $X$. For more details about almost periodic functions on topological groups, see the research monograph \cite{188} by B. M. Levitan.
\end{rem}
\begin{rem} From Remark \ref{ludar}, we may observe that for a given continuous function $f:G\to Y$ to be almost automorphic, it is important to start with any sequence in the full domain $G$; in this sense the functions presented in Examples \ref{exa01} and \ref{ExMulti} are not almost automorphic in the classical sense, but they are $({\mathrm R},{\mathcal B})$-multi-almost automorphic for concrete not empty sets $\mathrm{R}$. This clarifies the differences between this new class of functions and the classical ones known in the literature (\cite{30}).
\end{rem}
\subsection{Further properties of $({\mathrm R},{\mathcal B})$-multi-almost automorphic functions}
In order to give more interesting results about $({\mathrm R},{\mathcal B})$-multi-almost automorphic functions, we introduce the following concept. 

\begin{defn}\label{ (R,B) uniformly continuous}
Let $F : {\mathbb R}^{n} \times X \rightarrow Y$. Then, the function $F(\cdot;\cdot)$ is called $({\mathrm R},{\mathcal B})$-uniformly continuous if and only if for every $B\in {\mathcal B}$ and for every two sequences $({\bf a}_{k}=(a_{k}^{1},a_{k}^{2},\cdots ,a_{k}^{n})) , ({\bf b}_{k}=(b_{k}^{1},b_{k}^{2},\cdot \cdot\cdot ,b_{k}^{n})) \in {\mathrm R}$ such that $ |{\bf a}_{k}-{\bf b}_{k} | \rightarrow 0$  we have 
\begin{align}
\sup_{x\in B}\| F \bigl({\bf a}_{k};x\bigr) - F\bigl({\bf b}_{k};x\bigr) \|_{Y} \rightarrow 0,
\end{align}
as $ k\rightarrow +\infty $.
\end{defn}
If the function $F$ is uniformly continuous (in the classical sense) in ${\bf t}$ uniformly on bounded subsets of $X$. Then, in particular, $F$ is $({\mathrm R},{\mathcal B})$-uniformly continuous with ${\mathrm R}$ is the collection of all sequences of $\mathbb{R}^n$ and ${\mathcal B}$ is the collection of all bounded subsets of $X$. As in the usual considerations of almost automorphic functions, we have:

\begin{prop}\label{eovako54321}\index{function!(compactly) $({\mathrm R},{\mathcal B})$-multi-almost automorphic}
Suppose that $F : {\mathbb R}^{n} \times X \rightarrow Y$ is a continuous function. Then 
the function $F(\cdot;\cdot)$ is $({\mathrm R},{\mathcal B})$-multi-almost automorphic if and only if for every $B\in {\mathcal B}$ and for every sequence $({\bf b}_{k}=(b_{k}^{1},b_{k}^{2},\cdot \cdot\cdot ,b_{k}^{n})) \in {\mathrm R}$ there exist a subsequence $({\bf b}_{k_{l}}=(b_{k_{l}}^{1},b_{k_{l}}^{2},\cdot \cdot\cdot , b_{k_{l}}^{n}))$ of $({\bf b}_{k})$ and a function
$F^{\ast} : {\mathbb R}^{n} \times X \rightarrow Y$ such that
\begin{align}\label{snajkamlada}
\lim_{l\rightarrow +\infty}\lim_{l\rightarrow +\infty} F\Bigl({\bf t} -\bigl(b_{k_{l}}^{1},\cdot \cdot\cdot, b_{k_{l}}^{n}\bigr)+\bigl(b_{k_{m}}^{1},\cdot \cdot\cdot, b_{k_{m}}^{n}\bigr);x\Bigr)=F({\bf t};x) ,
\end{align}
pointwise for all $x\in B$ and ${\bf t}\in {\mathbb R}^{n}.$ 
\end{prop}

It is clear that we have the following simple results:

\begin{prop}\label{multi33}
Suppose that $F : {\mathbb R}^{n} \times X \rightarrow Y$ is a continuous function. 
If ${\mathcal B}'$ is a certain collection of subsets of $X$ which contains ${\mathcal B},$ 
${\mathrm R}'$ is a certain collection of sequences in ${\mathbb R}^{n}$ which contains ${\mathrm R}$
and $F(\cdot ;\cdot)$ is (compactly) $({\mathrm R}',{\mathcal B}')$-multi-almost automorphic, then 
 $F(\cdot ;\cdot)$ is (compactly) $({\mathrm R},{\mathcal B})$-multi-almost automorphic.
\end{prop}

\begin{prop}\label{superstebag}
Suppose that $F : {\mathbb R}^{n} \times X \rightarrow Y$ is $({\mathrm R},{\mathcal B})$-multi-almost automorphic and $\phi : Y \rightarrow Z$ is continuous.
Then $\phi \circ F : {\mathbb R}^{n} \times X \rightarrow Z$ is $({\mathrm R},{\mathcal B})$-multi-almost automorphic.
\end{prop}

In the next result we give a very useful characterization of compactly $({\mathrm R},{\mathcal B})$-multi-almost automorphic functions. A similar result was already established in \cite[Lemma 3.7]{Es-sebbar} for the particular case of $(\mathbb{R},\{0_{X}\})$-one-dimensional almost automorphic functions. 

\begin{thm}
Let  $F : {\mathbb R}^{n} \times X \rightarrow Y$ be continuous. Then, $F$ is compactly $({\mathrm R},{\mathcal B})$-multi-almost automorphic if and only $F$ is $({\mathrm R},{\mathcal B})$-multi-almost automorphic and $({\mathrm R},{\mathcal B})$-uniformly continuous.
\end{thm}
\begin{proof}
Let $ ({\mathrm R},{\mathcal B}) $ be fixed. Assume that $F$ is $({\mathrm R},{\mathcal B})$-multi-almost automorphic and $({\mathrm R},{\mathcal B})$-uniformly continuous. Then for every sequence $({\bf b}_{k}=(b_{k}^{1},b_{k}^{2},\cdot \cdot\cdot ,b_{k}^{n})) \in {\mathrm R}$ there exist a subsequence $({\bf b}_{k_{l}}=(b_{k_{l}}^{1},b_{k_{l}}^{2},\cdot \cdot\cdot , b_{k_{l}}^{n}))$ of $({\bf b}_{k})$ and a function
$F^{\ast} : {\mathbb R}^{n} \times X \rightarrow Y$ such that for each ${\bf t}\in {\mathbb R}^{n}$, we have
\begin{align}\label{love12345678}
\lim_{l\rightarrow +\infty}F\bigl({\bf t} +{\bf b}_{k_{l}};x\bigr)=F^{\ast}({\bf t};x) 
\end{align}
and
\begin{align}\label{love123456789}
\lim_{l\rightarrow +\infty}F^{\ast}\bigl({\bf t} -{\bf b}_{k_{l}};x\bigr)=F({\bf t};x),
\end{align}
uniformly for all $x\in B$. Let $x \in B$ and define the sequence 
$$   F_l (\cdot; x)=F(\cdot +{\bf b}_{k_{l}};x), \quad l \in \mathbb{N}.  $$ 
Then, from the $({\mathrm R},{\mathcal B})$-uniform continuity it follows that the family $ (F_l (\cdot;x))_l $ is equicontinuous. Hence, the convergence \eqref{love12345678} implies that $F^{\ast}({\cdot};x)$ is continuous. So, the convergence \eqref{love12345678} holds uniformly in compact subsets of  $\mathbb{R}^n $. Indeed, let $ \varepsilon >0 $ and let $I_c \subset \mathbb{R}^n$ be any compact subset. That is, there exist $ \delta >0 $ and $ \lbrace {\bf t}_{1},\cdots,{\bf t}_{m} \rbrace \subset I_c$ such that $I_c \subset \cup_{i=1}^{m} B({\bf t}_{i},\delta)  $. Thus, for all $ {\bf t} \in I_c $, there exists a certain  $i_0 \in {1,\cdots,m }$ satisfying $|{\bf t}-{\bf t}_{i_0}  | < \delta $. Then,  for all $ l $, we get
\begin{align*}
\| F_l ({\bf t}; x)- F^{\ast}({\bf t};x)\| & \leq \| F_l ({\bf t}; x)- F_l ({\bf t}_{i_0};x)\|+\| F_l ({\bf t}_{i_0}; x)- F^{\ast}({\bf t}_{i_0};x)\| \\
&+ \| F^{\ast}({\bf t}_{i_0};x)- F^{\ast}({\bf t};x)\|.
\end{align*}
The equicontinuity of $ (F_l (\cdot;x))_l $ along with the continuity of $F^{\ast}(\cdot;x)$ yield for $l$ large enough that
$$ \sup_{x\in B} \sup_{{\bf t} \in I_c} \| F_l ({\bf t}; x)- F^{\ast}({\bf t};x)\|  \rightarrow 0 \quad \text{ as } l\rightarrow \infty .$$
Similarly, we can prove that the limit in \eqref{love123456789} holds uniformly in compact subsets of $ \mathbb{R}^n $ using the $({\mathrm R},{\mathcal B})$-uniform continuity of $ F^{\ast}(\cdot;x)$ which is a consequence of \eqref{love12345678}. This proves the sufficiency.\\

Reciprocally, assume that $F$ is compactly $({\mathrm R},{\mathcal B})$-multi-almost automorphic. Then, it is $({\mathrm R},{\mathcal B})$-multi-almost automorphic. Therefore, it suffices to prove that $F$ is $({\mathrm R},{\mathcal B})$-uniformly continuous. Let $({\bf a}_{k}=(a_{k}^{1},a_{k}^{2},\cdots ,a_{k}^{n})) , ({\bf b}_{k}=(b_{k}^{1},b_{k}^{2},\cdot \cdot\cdot ,b_{k}^{n})) \in {\mathrm R}$ be any two sequences such that $ |{\bf a}_{k}-{\bf b}_{k} | \rightarrow 0$.  Then, we will prove that 
\begin{align*}
\sup_{x\in B}\| F \bigl({\bf a}_{k};x\bigr) - F\bigl({\bf b}_{k};x\bigr) \|_{Y} \rightarrow 0, 
\end{align*}
as $ k\rightarrow \infty $. Take any subsequences $ ({\bf a}_{k_{l}})_{l} $ and $ ({\bf b}_{k_{l}})_{l} $ of $ ({\bf a}_{k})_k  $ and $ ( {\bf b}_{k})_k  $ respectively. By the compact $({\mathrm R},{\mathcal B})$-multi-almost automorphy, there exist another two subsequences (still denoted by) $ ({\bf a }_{k_{l}})_l $ and $ ({\bf b }_{k_{l}})_l $ and a continuous function $F^{\ast}$ such that \eqref{love12345678} holds uniformly in compact subsets of $ \mathbb{R}^n $. Define $ {\bf c}_{k_{l}}:={\bf a}_{k_{l}}-{\bf b}_{k_{l}}, $ $l\in \mathbb{N} $ (by assumption $ | {\bf c}_{k_{l}}|\rightarrow 0 $ as $ l\rightarrow \infty $) and let $I_c \subset \mathbb{R}^n$ be compact such that  ${\bf c}_{k_{l}} \in I_c $ for all $ l \in \mathbb{N}$. Then, we have
\begin{align*}
&\| F \bigl({\bf a}_{k_{l}};x\bigr) - F\bigl({\bf b}_{k_{l}};x\bigr) \|_{Y} \leq \\
&\| F \bigl({\bf c}_{k_{l}}+{\bf b}_{k_{l}};x\bigr) - F^{\ast}\bigl({\bf c}_{k_{l}};x\bigr) \|_{Y} + \|F^{\ast} \bigl(0;x\bigr)- F^{\ast}\bigl({\bf c}_{k_{l}} ;x\bigr)  \|_{Y} \\ &+ \|F \bigl({\bf b}_{k_{l}};x\bigr)- F^{\ast}\bigl(0;x\bigr)  \|_{Y} \\
&\leq \sup_{{\bf t}\in I_c}\| F \bigl({\bf t}+{\bf b}_{k_{l}};x\bigr) - F^{\ast}\bigl({\bf t};x\bigr) \|_{Y} + \|F^{\ast} \bigl(0;x\bigr)- F^{\ast}\bigl({\bf c}_{k_{l}} ;x\bigr)  \|_{Y} \\ &+ \|F \bigl({\bf b}_{k_{l}};x\bigr)- F^{\ast}\bigl(0;x\bigr)  \|_{Y}.
\end{align*}
So, using  the compact almost automorphy of $F$ and the continuity of $F^{\ast} \bigl(\cdot;x\bigr)$, we obtain that 
$$\sup_{x\in B} \| F \bigl({\bf a}_{k_{l}};x\bigr) - F\bigl({\bf b}_{k_{l}};x\bigr) \|_{Y} \rightarrow 0 \quad \text{ as } l\rightarrow \infty.$$
Since a subsequence of any subsequence of $ \left( \sup_{x\in B} \| F \bigl({\bf a}_{k};x\bigr) - F\bigl({\bf b}_{k};x\bigr) \|_{Y}\right) _{k} $ converges. Then (the whole sequence), 
$$\sup_{x\in B} \| F \bigl({\bf a}_{k};x\bigr) - F\bigl({\bf b}_{k};x\bigr) \|_{Y} \rightarrow 0 \quad \text{ as } k \rightarrow \infty.$$
This proves the statement. 
\end{proof}

In \cite[Lemma 3.9.9]{nova-mono}, we have clarified the supremum formula for almost automorphic functions. This formula can be extended in our framework as follows:

\begin{prop}\label{superste} (The supremum formula)
Let $F : {\mathbb R}^{n} \times X \rightarrow Y$ be $({\mathrm R},{\mathcal B})$-multi-almost automorphic. Suppose that  there exist $a\geq 0$ and a sequence $b(\cdot)$ in ${\mathrm R}$ whose any subsequence is unbounded. 
Then we have
\begin{align}\label{m91}
\sup_{{\bf t}\in {\mathbb R}^{n},x\in X}\bigl\|F({\bf t};x) \Bigr\|_{Y}=\sup_{{\bf t}\in {\mathbb R}^{n},|t|\geq a,x\in X}\bigl\|F({\bf t};x) \Bigr\|_{Y}.
\end{align}
\end{prop}

\begin{proof}
We will include all relevent details of the proof for the sake of completeness. Let $\epsilon>0,$ $a\geq 0$ and $x\in X$ be given. Then \eqref{m91} will hold if we prove that 
\begin{align}\label{lozoh}
\bigl\|F({\bf t};x) \Bigr\|_{Y} \leq \epsilon+\sup_{{\bf t}\in {\mathbb R}^{n},|t|\geq a}\bigl\|F({\bf t};x) \Bigr\|_{Y}.
\end{align}
By assumption, there exists $B\in {\mathcal B}$ with $x\in B.$ Let $b(\cdot)$ be a sequence in ${\mathrm R}$ whose any subsequence is unbounded. Then we have \eqref{snajkamlada}, and consequently, there exist two integers $l_{0} \in {\mathbb N}$ and
$m_{0} \in {\mathbb N}$ such that
$$
\bigl\|F({\bf t};x) \Bigr\|_{Y} \leq \epsilon+\Bigl\| F\Bigl({\bf t}-\bigl(b_{k_{l}}^{1},\cdot \cdot\cdot, b_{k_{l}}^{n}\bigr)+\bigl(b_{k_{m}}^{1},\cdot \cdot\cdot, b_{k_{m}}^{n}\bigr); x\Bigr)\Bigr\|_{Y},\quad l\geq l_{0},\ m\geq m_{0}.
$$
In particular,
$$
\bigl\|F({\bf t};x) \Bigr\|_{Y} \leq \epsilon+\Bigl\| F\Bigl({\bf t}-\bigl(b_{k_{l_{0}}}^{1},\cdot \cdot\cdot, b_{k_{l_{0}}}^{n}\bigr)+\bigl(b_{k_{m}}^{1},\cdot \cdot\cdot, b_{k_{m}}^{n}\bigr) ; x\Bigr)\Bigr\|_{Y},\quad m\geq m_{0}.
$$
Since the sequence $((b_{k_{m}}^{1},\cdot \cdot\cdot, b_{k_{m}}^{n})_{m\geq m_{0}} $ is unbounded, \eqref{lozoh} follows immediately.
\end{proof}

Now we will prove also the following result.

\begin{prop}\label{2.1.10}
Suppose that for each integer $j\in {\mathbb N}$ the function $F_{j}(\cdot ; \cdot)$ is (compactly) $({\mathrm R},{\mathcal B})$-multi-almost automorphic. If
the sequence $(F_{j}(\cdot ;\cdot))$ converges uniformly to a function $F(\cdot ;\cdot)$, then the function $F(\cdot ;\cdot)$ is (compactly) $({\mathrm R},{\mathcal B})$-multi-almost automorphic.
\end{prop}

\begin{proof}
The proof is very similar to the proof of \cite[Theorem 2.1.10]{gaston} but we will provide all relevant details.
Let the set $B\in {\mathcal B}$ and the sequence $({\bf b}_{k}=(b_{k}^{1},b_{k}^{2},\cdot \cdot\cdot ,b_{k}^{n}))\in {\mathrm R}$ be given. Using the diagonal procedure, we get the existence of 
a subsequence $({\bf b}_{k_{l}}=(b_{k_{l}}^{1},b_{k_{l}}^{2},\cdot \cdot\cdot , b_{k_{l}}^{n}))$ of $({\bf b}_{k})$ such that for each integer $j\in {\mathbb N}$ there exists a function
$F^{\ast}_{j} : {\mathbb R}^{n} \times X \rightarrow Y$ such that
\begin{align}\label{metalac12345}
\lim_{l\rightarrow +\infty}\Bigl\| F_{j}\bigl({\bf t} +(b_{k_{l}}^{1},\cdot \cdot\cdot, b_{k_{l}}^{n});x\bigr)-F^{\ast}_{j}({\bf t};x) \Bigr\|_{Y}=0
\end{align}
and
\begin{align}\label{metalac123456}
\lim_{l\rightarrow +\infty}\Bigl\| F^{\ast}_{j}\bigl({\bf t} -(b_{k_{l}}^{1},\cdot \cdot\cdot, b_{k_{l}}^{n});x\bigr)-F_{j}({\bf t};x) \Bigr\|_{Y}=0,
\end{align}
pointwise for all $x\in B$ and ${\bf t}\in {\mathbb R}^{n}.$  Moreover, the  convergence of the above limits is uniform on compact subsets of ${\mathbb R}^{n}$ if 
for each integer $j\in {\mathbb N}$ the function $F_{j}(\cdot ; \cdot)$ is compactly $({\mathrm R},{\mathcal B})$-multi-almost automorphic.
Fix now an element $x\in B,$ a number ${\bf t}\in {\mathbb R}^{n}$ and positive real number $\epsilon>0.$
Since 
\begin{align*}
\Bigl\| F_{i}^{\ast}({\bf t} ; x) &-F_{j}^{\ast}({\bf t} ; x)\Bigr\|_{Y}\leq \Bigl\| F_{i}^{\ast}({\bf t} ; x)-F_{i}({\bf t}+(b_{k_{l}}^{1},\cdot \cdot\cdot, b_{k_{l}}^{n}) ; x)\Bigr\|_{Y}
\\&+\Bigl\| F_{i}({\bf t} +(b_{k_{l}}^{1},\cdot \cdot\cdot, b_{k_{l}}^{n}); x)-F_{j}({\bf t} +(b_{k_{l}}^{1},\cdot \cdot\cdot, b_{k_{l}}^{n}); x)\Bigr\|_{Y}
\\&+\Bigl\| F_{j}({\bf t} +(b_{k_{l}}^{1},\cdot \cdot\cdot, b_{k_{l}}^{n}); x)-F_{j}^{\ast}({\bf t} ; x)\Bigr\|_{Y},
\end{align*}
and \eqref{metalac12345}-\eqref{metalac123456} hold, we can find a number $l_{0}\in {\mathbb N}$ such that for all integers $l\geq l_{0}$ we have:
\begin{align*}
\Bigl\| F_{i}^{\ast}({\bf t} ; x)-F_{i}({\bf t}+(b_{k_{l}}^{1},\cdot \cdot\cdot, b_{k_{l}}^{n}) ; x)\Bigr\|_{Y}+\Bigl\| F_{j}({\bf t} +(b_{k_{l}}^{1},\cdot \cdot\cdot, b_{k_{l}}^{n}); x)-F_{j}^{\ast}({\bf t} ; x)\Bigr\|_{Y}<2\epsilon/3.
\end{align*}
Since the sequence $(F_{j}(\cdot ;\cdot))$ converges uniformly to a function $F(\cdot ;\cdot),$ there exists $N(\epsilon)\in {\mathbb N}$ such that
for all integers $i,\ j\in {\mathbb N}$ with $\min(i,j)\geq N(\epsilon)$ we have
$$
\Bigl\| F_{i}({\bf t} +(b_{k_{l}}^{1},\cdot \cdot\cdot, b_{k_{l}}^{n}); x)-F_{j}({\bf t} +(b_{k_{l}}^{1},\cdot \cdot\cdot, b_{k_{l}}^{n}); x)\Bigr\|_{Y}<\epsilon/3.
$$
This implies that $(F_{j}^{\ast}({\bf t} ;x))$ is a Cauchy sequence in $Y$ and therefore convergent to an element  $F^{\ast}({\bf t} ;x),$ say (the above arguments also shows that $\lim_{j\rightarrow +\infty}F_{j}^{\ast}({\bf t} ;x)=F^{\ast}({\bf t} ;x)$ uniformly on compact subsets of ${\mathbb R}^{n},$ for fixed element $x$).
It remains to be proved that \eqref{love12345678}-\eqref{love123456789} hold true. Toward this end,
observe that for each $j\in {\mathbb N}$ we have:
\begin{align*}
\Bigl\| F\bigl({\bf t} &+(b_{k_{l}}^{1},\cdot \cdot\cdot, b_{k_{l}}^{n});x\bigr)-F^{\ast}({\bf t};x) \Bigr\|_{Y}
\\& \leq \Bigl\| F\bigl({\bf t} +(b_{k_{l}}^{1},\cdot \cdot\cdot, b_{k_{l}}^{n});x\bigr)-F_{j}\bigl({\bf t} +(b_{k_{l}}^{1},\cdot \cdot\cdot, b_{k_{l}}^{n});x\bigr) \Bigr\|_{Y}
\\& +\Bigl\| F_{j}\bigl({\bf t} +(b_{k_{l}}^{1},\cdot \cdot\cdot, b_{k_{l}}^{n});x\bigr)-F^{\ast}_{j}({\bf t};x) \Bigr\|_{Y}
+ \Bigl\| F^{\ast}_{j}({\bf t};x) - F^{\ast}({\bf t};x) \Bigr\|_{Y}.
\end{align*}
It can be simply shown that there exists a number $j_{0}(\epsilon)\in {\mathbb N}$ such that for all integers $j\geq j_{0}$ we have
that the first and the third addend in the above estimate is less or greater than $\epsilon/3.$ For the second addend, take any integer $l\in {\mathbb N}$
such that
$$
\Bigl\| F_{j}\bigl({\bf t} +(b_{k_{l}}^{1},\cdot \cdot\cdot, b_{k_{l}}^{n});x\bigr)-F^{\ast}_{j}({\bf t};x) \Bigr\|_{Y}<\epsilon/3.
$$
This completes the proof in a routine manner.  
\end{proof}

For the sequel, we need the following general definition:

\begin{defn}\label{kompleks12345}
Suppose that 
${\mathbb D} \subseteq \Omega \subseteq {\mathbb R}^{n}$ and the set ${\mathbb D}$  is unbounded. By $C_{0,{\mathbb D},{\mathcal B}}(\Omega\times X :Y)$ we denote the vector space consisting of all continuous functions $Q : \Omega \times X \rightarrow Y$ such that, for every $B\in {\mathcal B},$ we have $\lim_{t\in {\mathbb D},|t|\rightarrow +\infty}Q({\bf t};x)=0,$ uniformly for $x\in B.$\index{space!$C_{0,{\mathbb D}}(\Omega \times X :Y)$}
\end{defn}

\begin{rem}\label{saznaoje}
In \cite[Definition 16-(i)]{dumath2}, we have used the following notion: Let $I=[0,\infty)$ or $I={\mathbb R}.$
By $C_{0,{\mathcal B}}(I \times X : Y)$ we denote the space of all
continuous functions $Q : I  \times X \rightarrow Y$ such that $\lim_{|t|\rightarrow \infty}Q(t, x) = 0$ uniformly for $x$ in any subset of $B\in {\mathcal B} .$ The notion introduced in Definition \ref{kompleks12345} properly generalizes this notion (see also case [L2.] considered above).
\end{rem}

Now we are ready to introduce the following notion:

\begin{defn}\label{braindamage12345}
Suppose that the set ${\mathbb D} \subseteq {\mathbb R}^{n}$ is unbounded, $i=1,2$ and
$F : {\mathbb R}^{n} \times X \rightarrow Y$ is a continuous function. Then we say that $F(\cdot ;\cdot)$ is 
${\mathbb D}$-asymptotically (compactly) $({\mathrm R},{\mathcal B})$-multi-almost automorphic
if and only if there exist a function $G(\cdot ;\cdot)$ which is (compactly) $({\mathrm R},{\mathcal B})$-multi-almost automorphic and a function
$Q\in C_{0,{\mathbb D},{\mathcal B}}({\mathbb R}^{n}\times X :Y)$ such that
$F({\bf t} ; x)=G({\bf t} ; x)+Q({\bf t} ; x)$ for all ${\bf t}\in {\mathbb R}^{n}$ and $x\in X.$

It is said that $F(\cdot ;\cdot)$ is 
asymptotically (compactly) $({\mathrm R},{\mathcal B})$-multi-almost automorphic if and only if $F(\cdot ;\cdot)$ is 
${\mathbb R}^{n}$-asymptotically (compactly) $({\mathrm R},{\mathcal B})$-multi-almost automorphic.
\index{function!${\mathbb D}$-asymptotically $({\mathrm R},{\mathcal B})$-multi-almost automorphic}
\index{function!compactly ${\mathbb D}$-asymptotically $({\mathrm R},{\mathcal B})$-multi-almost automorphic}
\end{defn}

The proof of the following proposition can be given as for the usually considered almost automorphic functions (\cite{gaston}):

\begin{prop}\label{kontinuitet}
\begin{itemize}
\item[(i)] Suppose that $c\in {\mathbb C}$ and $F(\cdot ; \cdot)$ is (compactly) $({\mathrm R},{\mathcal B})$-multi-almost automorphic.
Then $cF(\cdot ; \cdot)$ is (compactly) $({\mathrm R},{\mathcal B})$-multi-almost automorphic. The same holds for ${\mathbb D}$-asymptotically (compactly) $({\mathrm R},{\mathcal B})$-multi-almost automorphic functions.
\item[(ii)] 
\begin{itemize}
\item[(a)]
Suppose that $\tau\in {\mathbb R}^{n}$ and $F(\cdot ; \cdot)$ is (compactly) $({\mathrm R},{\mathcal B})$-multi-almost automorphic.
Then $F(\cdot +\tau; \cdot)$ is (compactly) $({\mathrm R},{\mathcal B})$-multi-almost automorphic. Furthermore, if $F(\cdot ; \cdot)$ is ${\mathbb D}$-asymptotically (compactly) $({\mathrm R},{\mathcal B})$-multi-almost automorphic,
then $F(\cdot +\tau; \cdot)$ is $({\mathbb D}-\tau)$-asymptotically (compactly) $({\mathrm R},{\mathcal B})$-multi-almost automorphic.
\item[(b)]  Suppose that $x_{0}\in X$ and $F(\cdot ; \cdot)$ is (compactly) $({\mathrm R},{\mathcal B})$-multi-almost automorphic.
Then $F(\cdot; \cdot+x_{0})$ is (compactly) $({\mathrm R},{\mathcal B}_{x_{0}})$-multi-almost automorphic. Furthermore, if $F(\cdot ; \cdot)$ is ${\mathbb D}$-asymptotically (compactly) $({\mathrm R},{\mathcal B})$-multi-almost automorphic, then 
$F(\cdot; \cdot+x_{0})$ is ${\mathbb D}$-asymptotically (compactly) $({\mathrm R},{\mathcal B}_{x_{0}})$-multi-almost automorphic.
\item[(c)] Suppose that $\tau\in {\mathbb R}^{n},$ $x_{0}\in X$ and $F(\cdot ; \cdot)$ is (compactly) $({\mathrm R},{\mathcal B})$-multi-almost automorphic.
Then $F(\cdot +\tau; \cdot +x_{0})$ is (compactly) $({\mathrm R},{\mathcal B}_{x_{0}})$-multi-almost automorphic, where ${\mathcal B}_{x_{0}}\equiv \{-x_{0}+B : B\in {\mathcal B}\}.$ Furthermore, if 
$F(\cdot ; \cdot)$ is ${\mathbb D}$-asymptotically (compactly) $({\mathrm R},{\mathcal B})$-multi-almost automorphic, then $F(\cdot +\tau; \cdot +x_{0})$ is $({\mathbb D}-\tau)$-asymptotically (compactly) $({\mathrm R},{\mathcal B}_{x_{0}})$-multi-almost automorphic.
\end{itemize}
\item[(iii)] \begin{itemize}
\item[(a)]
Suppose that $c\in {\mathbb C} \setminus \{0\}$ and $F(\cdot ; \cdot)$ is (compactly) $({\mathrm R},{\mathcal B})$-multi-almost automorphic.
Then $F(c \cdot ; \cdot)$ is (compactly) $({\mathrm R}_{c},{\mathcal B})$-multi-almost automorphic, where ${\mathrm R}_{c}\equiv \{c^{-1}{\bf b} : {\bf b}\in {\mathrm R}\}.$ Furthermore, if $F(\cdot ; \cdot)$ is ${\mathbb D}$-asymptotically (compactly) $({\mathrm R},{\mathcal B})$-multi-almost automorphic, then
$F(c \cdot ; \cdot)$ is ${\mathbb D}/c$-asymptotically (compactly) $({\mathrm R}_{c},{\mathcal B})$-multi-almost automorphic.
\item[(b)] Suppose that $c_{2}\in {\mathbb C}\setminus \{0\},$ and $F(\cdot ; \cdot)$ is (compactly) $({\mathrm R},{\mathcal B})$-multi-almost automorphic.
Then $F(\cdot ; c_{2}\cdot)$ is (compactly) $({\mathrm R},{\mathcal B}_{c_{2}})$-multi-almost automorphic, where ${\mathcal B}_{c_{2}}\equiv \{c_{2}^{-1}B : B\in {\mathcal B}\}.$ 
Furthermore, if $F(\cdot ; \cdot)$ is ${\mathbb D}$-asymptotically (compactly) $({\mathrm R},{\mathcal B})$-multi-almost automorphic, then
 $F(\cdot ; c_{2}\cdot)$ is  ${\mathbb D}$-asymptotically (compactly) $({\mathrm R},{\mathcal B}_{c_{2}})$-multi-almost automorphic.
\item[(c)]  Suppose that $c_{1}\in {\mathbb C}\setminus \{0\},$ $c_{2}\in {\mathbb C}\setminus \{0\},$ and $F(\cdot ; \cdot)$ is (compactly) $({\mathrm R},{\mathcal B})$-multi-almost automorphic.
Then $F(c_{1}\cdot ; c_{2}\cdot)$ is (compactly) $({\mathrm R}_{c_{1}},{\mathcal B}_{c_{2}})$-multi-almost automorphic.
Furthermore, if $F(\cdot ; \cdot)$ is ${\mathbb D}$-asymptotically (compactly) $({\mathrm R}_{c_{1}},{\mathcal B})$-multi-almost automorphic, then 
$F(c_{1}\cdot ; c_{2}\cdot)$ is ${\mathbb D}/c_{1}$-asymptotically (compactly) $({\mathrm R}_{c_{1}},{\mathcal B}_{c_{2}})$-multi-almost automorphic.
\end{itemize}
\item[(iv)] Suppose that $\alpha,\ \beta \in {\mathbb C}$ and, for every sequence which belongs to ${\mathrm R},$ we have that any its subsequence belongs to ${\mathrm R}.$ If $F(\cdot ; \cdot)$ and $G(\cdot ; \cdot)$ are (compactly) $({\mathrm R},{\mathcal B})$-multi-almost automorphic, then $\alpha F(\cdot ; \cdot)+\beta G(\cdot ; \cdot)$ is also (compactly) $({\mathrm R},{\mathcal B})$-multi-almost automorphic. The same holds for ${\mathbb D}$-asymptotically (compactly) $({\mathrm R},{\mathcal B})$-multi-almost automorphic functions.
\item[(v)] If $X\in {\mathcal B}$ and $F(\cdot ;\cdot)$ is 
asymptotically $({\mathrm R},{\mathcal B})$-multi-almost automorphic, then $F(\cdot ; \cdot)$ is bounded in case \emph{[L3]}.
\end{itemize}
\end{prop}

Using Proposition \ref{kontinuitet}(iv) and the supremum formula
clarified in Proposition \ref{superste} (see also the estimate \eqref{lozoh}), we can simply deduce that the decomposition in Definition \ref{braindamage12345} is unique:

\begin{prop}\label{okjerade}
Suppose that there exist a function $G_{i}(\cdot ;\cdot)$ which is $({\mathrm R},{\mathcal B})$-multi-almost automorphic and a function
$Q_{i}\in C_{0,{\mathbb R}^{n},{\mathcal B}}({\mathbb R}^{n}\times X :Y)$ such that
$F({\bf t} ; x)=G_{i}({\bf t} ; x)+Q_{i}({\bf t} ; x)$ for all ${\bf t}\in {\mathbb R}^{n}$ and $x\in X$ ($i=1,2$).
Then we have $G_{1}\equiv G_{2}$ and $Q_{1}\equiv Q_{2},$ provided that the collection ${\mathrm R}$ satisfies the following two conditions:
\begin{itemize}
\item[D1.] There exists a sequence in ${\mathrm R}$ whose any subsequence is unbounded.
\item[D2.] For every sequence which belongs to ${\mathrm R},$ we have that any its subsequence belongs to ${\mathrm R}.$
\end{itemize}
\end{prop}

Furthermore, arguing as in the proof of \cite[Lemma 4.28]{diagana}, we may deduce the following:

\begin{lem}\label{obazac}
Suppose that there exist an $({\mathrm R},{\mathcal B})$-multi-almost automorphic function $G(\cdot ;\cdot)$ and a function
$Q\in C_{0,{\mathbb R}^{n},{\mathcal B}}({\mathbb R}^{n}\times X :Y)$ such that
$F({\bf t} ; x)=G({\bf t} ; x)+Q({\bf t} ; x)$ for all ${\bf t}\in {\mathbb R}^{n}$ and $x\in X.$ 
Then we have 
$$
\overline{\bigl\{G({\bf t};x) : {\bf t}\in {\mathbb R}^{n},\ x\in X \bigr\}} \subseteq  \overline{\bigl\{F({\bf t};x) : {\bf t}\in {\mathbb R}^{n},\ x\in X \bigr\}},
$$
provided that condition \emph{[D1.]} holds.
\end{lem}

\subsection{Integrals and  derivatives of  $({\mathrm R},{\mathcal B})$-multi-almost automorphic functions}
Concerning the integral invariance of space consisting of all $({\mathrm R},{\mathcal B})$-multi-almost automorphic functions, we would like to state the following result:

\begin{prop}\label{convdiaggas}
Suppose that $h\in L^{1}({\mathbb R}^{n})$ as well as the function $F(\cdot ; \cdot)$ is bounded and $({\mathrm R},{\mathcal B})$-multi-almost automorphic. 
Then the function 
$$
(h\ast F)({\bf t};x):=\int_{{\mathbb R}^{n}}h(\sigma)F({\bf t}-\sigma;x)\, d\sigma,\quad {\bf t}\in {\mathbb R}^{n},\ x\in X
$$
is bounded and $({\mathrm R},{\mathcal B})$-multi-almost automorphic.  
\end{prop}

\begin{proof}
Since $h\in L^{1}({\mathbb R}^{n})$ and the function $F(\cdot ; \cdot)$ is bounded $({\mathrm R},{\mathcal B})$-multi-almost automorphic, 
it is clear that the function $(h\ast F)(\cdot ;\cdot)$ is well defined and bounded as well as that
the limit function $F^{\ast}(\cdot;\cdot)$ from Definition \ref{eovako} 
is bounded and measurable. 
The continuity of function $(h\ast F)(\cdot ;\cdot)$ follows from the dominated convergence theorem. 
The remainder of proof can be deduced as in the proof of \cite[Theorem 4.5, pp. 113--115]{diagana}; let us only note that the limit function for
 $(h\ast F)(\cdot ;\cdot)$ will be $(h\ast F^{\ast})(\cdot ;\cdot)$. 
\end{proof}

\begin{prop}\label{2.1.10obazac}
Suppose that conditions \emph{[D1.]-[D2.]} hold and for each integer $j\in {\mathbb N}$ the function $F_{j}(\cdot ; \cdot)$ is asymptotically (compactly) $({\mathrm R},{\mathcal B})$-multi-almost automorphic. If
the sequence $(F_{j}(\cdot ;\cdot))$ converges uniformly to a function $F(\cdot ;\cdot)$, then the function $F(\cdot ;\cdot)$ is asymptotically (compactly) $({\mathrm R},{\mathcal B})$-multi-almost automorphic.
\end{prop}

\begin{proof}
Due to Proposition \ref{okjerade}, 
we know that there exist a uniquely determined function $G(\cdot ;\cdot)$ which is $({\mathrm R},{\mathcal B})$-multi-almost automorphic and a uniquely determined function
$Q\in C_{0,{\mathbb R}^{n},{\mathcal B}}({\mathbb R}^{n}\times X :Y)$ such that
$F({\bf t} ; x)=G({\bf t} ; x)+Q({\bf t} ; x)$ for all ${\bf t}\in {\mathbb R}^{n}$ and $x\in X$. Furthermore, we have
$$
F_{j}({\bf t} ; x)-F_{m}({\bf t} ; x)=\bigl[ G_{j}({\bf t} ; x)-G_{m}({\bf t} ; x)\bigr]+\bigl[Q_{j}({\bf t} ; x)-Q_{m}({\bf t} ; x) \bigr],
$$
 for all ${\bf t}\in {\mathbb R}^{n},$ $x\in X$ and $j,\ m\in {\mathbb N}.$ Due to Proposition \ref{kontinuitet}(iv), we have that the function $F_{j}(\cdot ;\cdot)-F_{m}(\cdot ;\cdot)$ is asymptotically $({\mathrm R},{\mathcal B})$-multi-almost automorphic as well as that
the function $G_{j}(\cdot ;\cdot)-G_{m}(\cdot ;\cdot)$ is $({\mathrm R},{\mathcal B})$-multi-almost automorphic ($j,\ m\in {\mathbb N}$). Keeping in mind this fact, Lemma \ref{obazac} and the argumentation used in the proof of \cite[Theorem 4.29]{diagana}, we get that
\begin{align*}
3\sup_{{\bf t} \in {\mathbb R}^{n},x\in X}&\Bigl\|  F_{j}({\bf t} ; x)-F_{m}({\bf t} ; x) \Bigr\|_{Y}
\\& \geq \sup_{{\bf t} \in {\mathbb R}^{n},x\in X}\Bigl\|  G_{j}({\bf t} ; x)-G_{m}({\bf t} ; x) \Bigr\|_{Y}+\sup_{{\bf t} \in {\mathbb R}^{n},x\in X}\Bigl\|  Q_{j}({\bf t} ; x)-Q_{m}({\bf t} ; x) \Bigr\|_{Y},
\end{align*}
for any $j,\ m\in {\mathbb N}.$
This implies that the sequences $(G_{j}(\cdot;\cdot))$ and $(Q_{j}(\cdot;\cdot))$ converge uniformly to the functions $G(\cdot;\cdot)$ and $Q(\cdot;\cdot),$ respectively. Due to Proposition \ref{2.1.10}, we get that the function $G(\cdot;\cdot)$ is $({\mathrm R},{\mathcal B})$-multi-almost automorphic.
The final conclusion follows from the obvious equality $F=G+Q$ and the fact that $C_{0,{\mathbb R}^{n},{\mathcal B}}({\mathbb R}^{n} \times X : Y)$ is a Banach space.
\end{proof}

Concerning the partial derivatives of (asymptotically) $({\mathrm R},{\mathcal B})$-multi-almost automorphic functions, we would like to state the following result (by $(e_{1},e_{2},\cdot \cdot \cdot,e_{n})$ we denote the standard basis of ${\mathbb R}^{n}$):

\begin{prop}\label{aljaska12345}
\begin{itemize}
\item[(i)] Suppose that the function $F(\cdot; \cdot)$ is (compactly) $({\mathrm R},{\mathcal B})$-multi-almost automorphic, \emph{[D2.]} holds, the partial derivative 
$$
\frac{\partial F(\cdot ; \cdot)}{\partial t_{i}}:=\lim_{h\rightarrow 0}\frac{F(\cdot +he_{i};\cdot)-F(\cdot; \cdot)}{h},\quad {\bf t}\in {\mathbb R}^{n},\ x\in X
$$
exists and it is uniformly continuous on ${\mathcal B},$ i.e.,
\begin{align*}
(\forall B\in {\mathcal B})& \ (\forall \epsilon>0) \ (\exists \delta > 0) \ (\forall {\bf t'},\ {\bf t''} \in {\mathbb R}^{n})\ (\forall x\in B)
\\& \Biggl( \bigl| {\bf t'}-{\bf t''} \bigr|<\delta \Rightarrow \Bigl\| \frac{\partial F({\bf t'} ; x)}{\partial t_{i}}-\frac{\partial F({\bf t''} ; x)}{\partial t_{i}}\Bigr\|<\epsilon  \Biggr).
\end{align*}
Then the function $
\frac{\partial F(\cdot ; \cdot)}{\partial t_{i}}$ is  (compactly) $({\mathrm R},{\mathcal B})$-multi-almost automorphic.
\item[(ii)] Suppose that the function $F(\cdot; \cdot)$ is asymptotically (compactly) $({\mathrm R},{\mathcal B})$-multi-almost automorphic, \emph{[D1.]-[D2.]} hold, the partial derivative 
$
\frac{\partial F({\bf t} ; x)}{\partial t_{i}}
$
exists for all ${\bf t}\in {\mathbb R}^{n}$, $ x\in X$ and it is uniformly continuous on ${\mathcal B}.$
Then the function $
\frac{\partial F(\cdot ; \cdot)}{\partial t_{i}}$ is asymptotically (compactly) $({\mathrm R},{\mathcal B})$-multi-almost automorphic. 
\end{itemize}
\end{prop}

\begin{proof}
We will prove only (i) because (ii) follows similarly, by appealing to Proposition \ref{2.1.10obazac} instead of 
Proposition \ref{2.1.10}. The proof immediately follows from the fact that the sequence $(F_{j}(\cdot;\cdot)\equiv j[F(\cdot +j^{-1}e_{i}; \cdot)-F(\cdot;\cdot)])$ of (compactly) $({\mathrm R},{\mathcal B})$-multi-almost automorphic functions converges uniformly to the function $\frac{\partial F(\cdot ; \cdot)}{\partial t_{i}}$ as $j\rightarrow +\infty.$ This can be shown as in one-dimesional case, by observing that
$$
F_{j}(\cdot;\cdot)-\frac{\partial F(\cdot ; \cdot)}{\partial t_{i}}=j\int^{1/j}_{0}\Biggl[ \frac{\partial F(\cdot +se_{i}; \cdot)}{\partial t_{i}}-\frac{\partial F(\cdot ; \cdot)}{\partial t_{i}}\Biggr]\, ds.
$$
\end{proof}
\subsection{Composition theorems for $({\mathrm R},{\mathcal B})$-multi-almost automorphic type functions}\label{marekzero}

Suppose that $F : {\mathbb R}^{n} \times X \rightarrow Y$ and $G : {\mathbb R}^{m} \times Y \rightarrow Z$ are given functions, where $m\in {\mathbb N}.$ The main aim of this subsection is to analyze the $({\mathrm R},{\mathcal B})$-multi-almost automorphic properties of the following multi-dimensional Nemytskii operator $H : {\mathbb R}^{n}  \times {\mathbb R}^{m} \times X \rightarrow Z$, given by
$$
H({\bf t};{\bf s}; x):=G\bigl({\bf s} ; F({\bf t}; x)\bigr),\quad {\bf t} \in {\mathbb R}^{n},\ {\bf s} \in {\mathbb R}^{m},\ x\in X.
$$

If $m=n,$ then we will also analyze the $({\mathrm R},{\mathcal B})$-multi-almost automorphic properties of the following multi-dimensional Nemytskii operator $W : {\mathbb R}^{n}  \times X \rightarrow Z,$ given by
$$
W({\bf t}; x):=G\bigl({\bf t} ; F({\bf t}; x)\bigr),\quad {\bf t} \in {\mathbb R}^{n},\ x\in X.
$$

We will first state the following generalization of \cite[Theorem 4.16]{diagana}; the proof is similar to the proof of the above-mentioned theorem but we will present it for the sake
of completeness:

\begin{thm}\label{eovakoonako}
Suppose that $F : {\mathbb R}^{n} \times X \rightarrow Y$ is $({\mathrm R},{\mathcal B})$-multi-almost automorphic and $G : {\mathbb R}^{n} \times X \rightarrow Y$ is $({\mathrm R}',{\mathcal B}')$-multi-almost automorphic,
where ${\mathrm R}'$ is a collection of all sequences $b : {\mathbb N} \rightarrow {\mathbb R}^{n}$ from ${\mathrm R}$ and all their subsequences, as well as
\begin{align}\label{ljuvenauto}
{\mathcal B}':=\Biggl\{\bigcup_{{\bf t}\in {\mathbb R}^{n}}F(t; B) : B\in {\mathcal B}\Biggr\}.
\end{align}
If there exists a finite constant $L>0$ such that
\begin{align}\label{zigzverinije}
\bigl\| G({\bf t} ;x) -G({\bf t} ;y)\bigr\|_{Z}\leq L\|x-y\|_{Y},\quad {\bf t}\in {\mathbb R}^{n},\ x,\ y\in Y,
\end{align}
then the function $W(\cdot; \cdot)$ is $({\mathrm R},{\mathcal B})$-multi-almost automorphic.
\end{thm}

\begin{proof}
Let the set $B\in {\mathcal B}$ and the sequence $({\bf b}_{k}=(b_{k}^{1},b_{k}^{2},\cdot \cdot\cdot ,b_{k}^{n})) \in {\mathrm R}$ be given.
By definition, there exist a subsequence $({\bf b}_{k_{l}}=(b_{k_{l}}^{1},b_{k_{l}}^{2},\cdot \cdot\cdot , b_{k_{l}}^{n}))$ of $({\bf b}_{k})$ and a function
$F^{\ast} : {\mathbb R}^{n} \times X \rightarrow Y$ such that \eqref{love12345678}-\eqref{love123456789} hold true. Set $B':=\bigcup_{{\bf t}\in {\mathbb R}^{n}}F(t; B)$ and $b':=({\bf b}_{k_{l}}).$
Then there exist a subsequence $({\bf b}_{k_{l_{m}}}=(b_{k_{l_{m}}}^{1},b_{k_{l_{m}}}^{2},\cdot \cdot\cdot , b_{k_{l_{m}}}^{n}))$ of $({\bf b}_{k_{l}})$ and a function
$G^{\ast} : {\mathbb R}^{n} \times X \rightarrow Y$ such that
\begin{align}\label{loveeman}
\lim_{m\rightarrow +\infty}\Bigl\| G\bigl({\bf t} +(b_{k_{l_{m}}}^{1},\cdot \cdot\cdot, b_{k_{l_{m}}}^{n});y\bigr)-G^{\ast}({\bf t};y) \Bigr\|_{Z}=0
\end{align}
and
\begin{align}\label{loveema}
\lim_{m\rightarrow +\infty}\Bigl\| G^{\ast}\bigl({\bf t} -(b_{k_{l_{m}}}^{1},\cdot \cdot\cdot, b_{k_{l_{m}}}^{n});y\bigr)-G({\bf t};y) \Bigr\|_{Z}=0,
\end{align}
pointwise for all $y\in B'$ and ${\bf t}\in {\mathbb R}^{n}.$
It suffices to show that
\begin{align}\label{emanuel}
\lim_{m\rightarrow +\infty}\Bigl\| G\bigl({\bf t} +(b_{k_{l_{m}}}^{1},\cdot \cdot\cdot, b_{k_{l_{m}}}^{n});F\bigl({\bf t} +(b_{k_{l_{m}}}^{1},\cdot \cdot\cdot, b_{k_{l_{m}}}^{n});x\bigr)\bigr)-G^{\ast}({\bf t};F^{\ast}({\bf t};x)) \Bigr\|_{Z}=0
\end{align}
and
\begin{align}\label{loveem}
\lim_{m\rightarrow +\infty}\Bigl\| G^{\ast}\bigl({\bf t} -(b_{k_{l_{m}}}^{1},\cdot \cdot\cdot, b_{k_{l_{m}}}^{n});F^{\ast}\bigl({\bf t} -(b_{k_{l_{m}}}^{1},\cdot \cdot\cdot, b_{k_{l_{m}}}^{n});x\bigr)\bigr)-G({\bf t};F({\bf t};x)) \Bigr\|_{Z}=0.
\end{align}
We will prove only \eqref{emanuel} since the proof of \eqref{loveem} is quite analogous. For simplicity, denote ${\bf \tau_{m}}:=(b_{k_{l_{m}}}^{1},\cdot \cdot\cdot, b_{k_{l_{m}}}^{n})$ for all $m\in {\mathbb N}$. We have (${\bf t}\in {\mathbb R}^{n},$ $x\in B,$ $m\in {\mathbb N}$):
\begin{align*}
&\Bigl\| G\bigl({\bf t} +{\bf \tau_{m}};F\bigl({\bf t} +{\bf \tau_{m}};x\bigr)\bigr)-G^{\ast}({\bf t};F^{\ast}({\bf t};x)) \Bigr\|_{Z}
\\& \leq  \Bigl\| G\bigl({\bf t} +{\bf \tau_{m}};F\bigl({\bf t} +{\bf \tau_{m}};x\bigr)\bigr)-G^{\ast}({\bf t}+{\bf \tau_{m}};F^{\ast}({\bf t};x)) \Bigr\|_{Z}
\\& + \Bigl\| G^{\ast}({\bf t}+{\bf \tau_{m}};F^{\ast}({\bf t};x)) -G^{\ast}({\bf t};F^{\ast}({\bf t};x)) \Bigr\|_{Z}
\\& \leq L\Bigl\|F\bigl({\bf t} +{\bf \tau_{m}};x\bigr)-F^{\ast}({\bf t};x) \Bigr\|_{Y}+\Bigl\| G^{\ast}({\bf t}+{\bf \tau_{m}};F^{\ast}({\bf t};x)) -G^{\ast}({\bf t};F^{\ast}({\bf t};x)) \Bigr\|_{Z}.
\end{align*}
Since $x\in B$ and $F^{\ast}({\bf t};x)\in B'$ for all ${\bf t}\in {\mathbb R}^{n}$, \eqref{emanuel} follows by applying \eqref{love12345678} and \eqref{emanuel}, which completes the proof of the theorem.
\end{proof}

\begin{rem}\label{animamajka}
Suppose that $F : {\mathbb R}^{n} \times X \rightarrow Y$ is compactly $({\mathrm R},{\mathcal B})$-multi-almost automorphic, $G : {\mathbb R}^{n} \times X \rightarrow Y$ is compactly $({\mathrm R}',{\mathcal B}')$-multi-almost automorphic and additionally satisfies that the limit equations
\eqref{love12345678}-\eqref{love123456789} for $G(\cdot; \cdot)$ and $G^{\ast}(\cdot ;\cdot)$ hold uniformly on $K\times B',$ where $K$ is an arbitrary compact subset of ${\mathbb R}^{n}$ and $B'$ is an arbitrary set in ${\mathcal B}'.$ By the above proof, the function
 $W(\cdot; \cdot)$ is then compactly $({\mathrm R},{\mathcal B})$-multi-almost automorphic.
\end{rem}

A slight modification of the proof of Theorem \ref{eovakoonako} (cf. also the proof of \cite[Theorem 4.17]{diagana}) shows that the following result holds true:

\begin{thm}\label{eovako12345onako}
Suppose that $F : {\mathbb R}^{n} \times X \rightarrow Y$ is $({\mathrm R},{\mathcal B})$-multi-almost automorphic and $G : {\mathbb R}^{n} \times X \rightarrow Y$ is $({\mathrm R}',{\mathcal B}')$-multi-almost automorphic,
where ${\mathrm R}'$ is a collection of all sequences $b : {\mathbb N} \rightarrow {\mathbb R}^{n}$ from ${\mathrm R}$ and all their subsequences, as well as
$
{\mathcal B}'$ be given by \eqref{ljuvenauto}.
Set
$$
{\mathcal B}^{'*}:=\Biggl\{\bigcup_{{\bf t}\in {\mathbb R}^{n}}F^{*}(t; B) : B\in {\mathcal B}\Biggr\}.
$$
If 
\begin{align*}
& (\forall  B \in {\mathcal B}) \ (\forall \epsilon >0) \ (\exists \delta >0) 
\\ & \Bigl( x,\ y \in  {\mathcal B}' \cup {\mathcal B}^{'*}\mbox{ and }\ \bigl\| x-y\bigr\|_{Y} <\delta \Rightarrow \bigl\| G({\bf t};x)-G({\bf t};y)\bigr\|_{Z}<\epsilon,\ {\bf t} \in {\mathbb R}^{n}\Bigr),
\end{align*}
then the function $W(\cdot; \cdot)$ is $({\mathrm R},{\mathcal B})$-multi-almost automorphic.
\end{thm}

Now we proceed with the analysis of composition theorems for asymptotically $({\mathrm R},{\mathcal B})$-multi-almost automorphic functions. Our first result corresponds to Theorem \ref{eovakoonako} and \cite[Theorem 4.34]{diagana}:

\begin{thm}\label{eovakoonako} 
Suppose that $F_{0} : {\mathbb R}^{n} \times X \rightarrow Y$ is $({\mathrm R},{\mathcal B})$-multi-almost automorphic,
$Q_{0}\in C_{0,{\mathbb R}^{n},{\mathcal B}}({\mathbb R}^{n} \times X : Y)$ and $F({\bf t} ; x)=F_{0}({\bf t} ; x)
+Q_{0}({\bf t} ; x)$ for all ${\bf t}\in {\mathbb R}^{n}$ and $x\in X.$ 
Suppose further that $G_{1} : {\mathbb R}^{n} \times X \rightarrow Y$ is $({\mathrm R}',{\mathcal B}')$-multi-almost automorphic,
where ${\mathrm R}'$ is a collection of all sequences $b : {\mathbb N} \rightarrow {\mathbb R}^{n}$ from ${\mathrm R}$ and all their subsequences as well as ${\mathcal B}'$ is defined by \eqref{ljuvenauto} with the function $F(\cdot;\cdot)$ replaced therein by the function
$F_{0}(\cdot;\cdot),$
$Q_{1}\in C_{0,{\mathbb R}^{n},{\mathcal B}_{1}}({\mathbb R}^{n} \times Y : Z),$
where
\begin{align}\label{objasni}
{\mathcal B}_{1}:=\Biggl\{\bigcup_{{\bf t}\in {\mathbb R}^{n}}F(t; B) : B\in {\mathcal B}\Biggr\},
\end{align}
and $G({\bf t} ; x)=G_{1}({\bf t} ; x)
+Q_{1}({\bf t} ; x)$ for all ${\bf t}\in {\mathbb R}^{n}$ and $x\in X.$
If there exists a finite constant $L>0$ such that the estimate \eqref{zigzverinije} holds with the function $G(\cdot;\cdot)$ replaced therein by the function
$G_{1}(\cdot;\cdot),$
then the function $W(\cdot; \cdot)$ is asymptotically $({\mathrm R},{\mathcal B})$-multi-almost automorphic.
\end{thm}

\begin{proof}
Using the above assumptions and Theorem \ref{eovakoonako}, we have that the function $({\bf t};x)\mapsto G_{1}({\bf t}; F_{0}({\bf t};x)),$ ${\bf t}\in {\mathbb R}^{n},$ $x\in X$ is $({\mathrm R},{\mathcal B})$-multi-almost automorphic.
Furthermore, we have the following decomposition
\begin{align*}
W({\bf t}; x)=G_{1}({\bf t}; F_{0}({\bf t};x))+\Bigl[ G_{1}({\bf t}; F({\bf t};x))-G_{1}({\bf t}; F_{0}({\bf t};x)) \Bigr]+Q_{1}({\bf t}; F({\bf t}; x)),
\end{align*}
for any ${\bf t}\in {\mathbb R}^{n}$ and $x\in X.$
Since 
$$
\Bigl\| G_{1}({\bf t}; F({\bf t};x))-G_{1}({\bf t}; F_{0}({\bf t};x)) \Bigr\|_{Z}\leq L \bigl\| Q_{0}({\bf t} ; x)\bigr\|_{Y},\quad  {\bf t}\in {\mathbb R}^{n},\ x\in X,
$$
we have that the function $({\bf t};x)\mapsto G_{1}({\bf t}; F({\bf t};x))-G_{1}({\bf t}; F_{0}({\bf t};x)),$ ${\bf t}\in {\mathbb R}^{n},$ $x\in X$ belongs to the space $C_{0,{\mathbb R}^{n},{\mathcal B}}({\mathbb R}^{n} \times X : Z).$ The same holds for the function 
$({\bf t};x)\mapsto Q_{1}({\bf t}; F({\bf t}; x)),$ ${\bf t}\in {\mathbb R}^{n},$ $x\in X$ because of our choice of the collection ${\mathcal B}_{1}$ in \eqref{objasni}. The proof of the theorem is thereby complete.
\end{proof}

Similarly we can prove the following result which corresponds to Theorem \ref{eovako12345onako} and \cite[Theorem 4.35]{diagana}:

\begin{thm}\label{eovakoonako}
Suppose that $F_{0} : {\mathbb R}^{n} \times X \rightarrow Y$ is $({\mathrm R},{\mathcal B})$-multi-almost automorphic,
$Q_{0}\in C_{0,{\mathbb R}^{n},{\mathcal B}}({\mathbb R}^{n} \times X : Y)$ and $F({\bf t} ; x)=F_{0}({\bf t} ; x)
+Q_{0}({\bf t} ; x)$ for all ${\bf t}\in {\mathbb R}^{n}$ and $x\in X.$ 
Suppose further that $G_{1} : {\mathbb R}^{n} \times X \rightarrow Y$ is $({\mathrm R}',{\mathcal B}')$-multi-almost automorphic,
where ${\mathrm R}'$ is a collection of all sequences $b : {\mathbb N} \rightarrow {\mathbb R}^{n}$ from ${\mathrm R}$ and all their subsequences as well as ${\mathcal B}'$ is defined by \eqref{ljuvenauto} with the function $F(\cdot;\cdot)$ replaced therein by the function
$F_{0}(\cdot;\cdot),$
$Q_{1}\in C_{0,{\mathbb R}^{n},{\mathcal B}_{1}}({\mathbb R}^{n} \times Y : Z),$
where
${\mathcal B}_{1}$ is given through \eqref{objasni},
and $G({\bf t} ; x)=G_{1}({\bf t} ; x)
+Q_{1}({\bf t} ; x)$ for all ${\bf t}\in {\mathbb R}^{n}$ and $x\in X.$
Set
$$
{\mathcal B}_{2}:=\Biggl\{\bigcup_{{\bf t}\in {\mathbb R}^{n}}F_{0}(t; B) : B\in {\mathcal B}\Biggr\} \cup \Biggl\{\bigcup_{{\bf t}\in {\mathbb R}^{n}}F_{0}^{\ast}(t; B) : B\in {\mathcal B}\Biggr\}.
$$
If 
\begin{align*}
& (\forall B \in {\mathcal B}) \ (\forall \epsilon >0) \ (\exists \delta >0) 
\\ & \Bigl( x,\ y \in  {\mathcal B}_{1} \cup {\mathcal B}_{2}\mbox{ and }\ \bigl\| x-y\bigr\|_{Y} <\delta \Rightarrow \bigl\| G_{1}({\bf t};x)-G_{1}({\bf t};y)\bigr\|_{Z}<\epsilon,\ {\bf t} \in {\mathbb R}^{n}\Bigr),
\end{align*}
then the function $W(\cdot; \cdot)$ is asymptotically $({\mathrm R},{\mathcal B})$-multi-almost automorphic.
\end{thm}

Composition theorems for (asymptotically) compactly $({\mathrm R},{\mathcal B})$-multi-almost automorphic functions will be considered somewhere else (see e.g., \cite[Lemma 4.36, Lemma 4.37, Lemma 4.38]{diagana} and \cite{sebbar} for the one-dimensional case).
\subsection{Convolution of $({\mathrm R},\mathcal{B})$-multi-almost automorphic functions}
For ${\bf t} = (t_1,t_2,\cdots , t_n)$, let us introduce the notation $\mathcal{I}_{{\bf t}}=(-\infty,t_1]\times (-\infty,t_2]\times \cdots \times (-\infty,t_n]$ and also impose the following condition:

\begin{itemize}
\item [ (E1)]  $\sup_{{\bf t}} \displaystyle\int_{\mathcal{I}_{\bf t}}K({\bf t -\eta})d \eta\, <\, +\infty\, .$
\end{itemize}

\begin{thm}\label{TI01} Let $f:\mathbb{R}^{n}\longrightarrow X$ be a ${\mathrm R}$-multi-almost automorphic function and $K$ is a kernel that satisfies the condition (E1). Then 
$$F({\bf t})=\int_{\mathcal{I}_{\bf t}}K({\bf t}-\eta)f(\eta)d\eta$$
yields a ${\mathrm R}$-multi-almost automorphic function.
\end{thm}
\begin{proof}
First of all observe that, using condition (E1) and the Lebesgue's  dominated convergence Theorem we have that  $F$ is a continuous function of ${\bf t}$. On the other hand, since $f$ is ${\mathrm R}$-multi-AA, given a sequence $b_n\in \mathrm{R}$, there exist a subsequence $c_n$ of $b_n$ and a function $\tilde{f}$ such that the following pointwise limits holds:
$$\lim_{n\to \infty}f({\bf t}+c_n)=\tilde{f}({\bf t}),$$
and
$$\lim_{n\to \infty} \tilde{f}({\bf t}-c_n)=f({\bf t}).$$
Now, let us take the (well-defined) function:
$$F^*({\bf t}):=\int_{\mathcal{I}_{\bf t}}K({\bf t}-\eta)\tilde{f}(\eta)d \eta \, .$$
Then, we have
\begin{eqnarray*}
\|F({\bf t}+c_n)-F^*(t)\| &= & \Big{|} \int_{\mathcal{I}_{{\bf t}+c_n}}K({\bf t}+c_n -\eta) f(\eta)d \eta-\int_{\mathcal{I}_{\bf t}}K({\bf t}-\eta)\tilde{f}(\eta)d \eta \Big{|}\\
&\leq & \int_{\mathcal{I}_{\bf t}}|K({\bf t}-\eta)| \| f(\eta +c_n)+\tilde{f}(\eta)\| d \eta \, .
\end{eqnarray*}
From which, using condition (E1) and Lebesgue's dominated convergence Theorem, we have 
$$\lim_{n\to \infty}F({\bf t}+c_n)=F^*(\bf t);$$
similarly, we obtain
$$\lim_{n\to \infty}F^*({\bf t}-c_n)=F(\bf t).$$
\end{proof}

Let $\mathbb{D}$ be an unbounded subset of $\mathbb{R}^n$ and $K$ a kernel function defined on $\mathbb{R}^n$. For the invariance of the space of $\mathrm{R}$-multi-asymptotically almost automorphic functions under some particular integral operators, we need the following additional conditions:

\begin{itemize}
\item [ (E2)]  $\displaystyle\lim_{|{\bf t}|\to +\infty, {\bf t}\in \mathbb{D}}\int_{\mathcal{I}_{\bf t} \cap \mathbb{D}^c}
K({\bf t -\eta})d \eta \, =0 \, .$
\item [ (E3)] for every $r>0$ (big enough), it is satisfied that
 $$\displaystyle\lim_{|{\bf t}|\to +\infty, {\bf t}\in \mathbb{D}}\int_{\mathcal{I}_{\bf t} \cap \mathbb{D} \cap B(0,r)}K({\bf t -\eta})d \eta \, =0 \, .$$
\end{itemize}


\begin{thm}\label{Theorem existence linear VIO}
Assume that $(E1),\, (E2)$ and $(E3)$ be fulfilled. Let $\mathbb{D}_{{\bf t}}=\mathcal{I}_{{\bf t}}\cap \mathbb{D}$ be such that $int(\mathbb{D}_{{\bf t_0}})\not = \emptyset$ for some ${\bf t_0}\in \mathbb{R}^n$. Then, the integral
$$  \Gamma f({\bf t})=\int\limits_{\mathbb{D}_{{\bf t}}} K({\bf t}-\eta)f(\eta)d \eta \, ,$$
defines an operator $\Gamma$ from the space of $\mathbb{D}$- asymptotically ${\mathrm{R}}$-multi-almost automorphic functions to itself.
\end{thm}
\begin{proof}
Let $f=f_a+f_0$, be a $\mathbb{D}$-asymptotically ${\mathrm{R}}$-multi-almost automorphic function, where $f_a$ is ${\mathrm R}$-multi-almost automorphic and $f_0$ is in $C_{0,\mathbb{D}}(\mathbb{R}^n; Y)$, then
\begin{eqnarray*}
\int\limits_{\mathbb{D}_{{\bf t}}} K({\bf t}-\eta)f(\eta)d \eta &=& \int\limits_{\mathbb{D}_{{\bf t}}} K({\bf t}-\eta)\left( f_a(\eta)+f_0(\eta)\right) d \eta \\
&=&\int\limits_{\mathcal{I}_{\bf t}} K({\bf t}-\eta) f_a(\eta)d (\eta)-
\int\limits_{\mathcal{I}_{\bf t}\cap \mathbb{D}^c}K({\bf t}-\eta) f_a(\eta)d (\eta)+\\
&+&\int\limits_{\mathbb{D}_{\bf t}} K({\bf t}-\eta)f_0(\eta) d \eta \\
&=&\Gamma_1f_a({\bf t})+\Gamma_2f({\bf t})\, ,
\end{eqnarray*}
where 
$$\Gamma_1f_a({\bf t}):=\int\limits_{\mathcal{I}_{{\bf t}}} K({\bf t}-\eta) f_a(\eta)d (\eta)\,  ,$$
and
$$\Gamma_2f({\bf t}):= \int\limits_{\mathbb{D}_{\bf t}} K({\bf t}-\eta)f_0(\eta) d \eta -\int\limits_{\mathcal{I}_{{\bf t}}\cap \mathbb{D}^c}K({\bf t}-\eta) f_a(\eta)d (\eta)\, .$$
From Theorem \ref{TI01} we have that $\Gamma_1f_a$ is $\mathrm{R}$-multi-almost automorphic. Now let us prove the following decay at infinity across $\mathbb{D}$:
$$\lim_{{\bf t} \in \mathbb{D}: |{\bf t}|\to +\infty}\Gamma_2f({\bf t}) =0\, .$$
First note that since $||f_a||_{\infty}<\infty$ and condition (E2) hods, then:
$$\lim_{{\bf t} \in \mathbb{D}: |{\bf t}|\to +\infty}\int\limits_{\mathcal{I}_{\bf t}\cap \mathbb{D}^c}K({\bf t}-\eta) f_a(\eta)d (\eta) =0\; ,$$ 
therefore, the second integral in the representation of $\Gamma_2 f$ vanishes at infinity, it rest to work with the first integral. Let $\epsilon >0$, since $f_0 \in C_{0,\mathbb{D}}(\Omega;Y)$ there exists $r>0$ (big enough) such that for all ${\bf t} \in \mathbb{D}$ with $|{\bf t}| >r$ we have $|f_0({\bf t})| <\epsilon$, moreover because of condition (E3),  for this $r$ and for $|{\bf t}|$ big enough we have:
$$\int\limits_{B(0,r)\cap \mathbb{D}_{{\bf t}}} K({\bf t}-\eta) d \eta <\epsilon \, ;$$ 
therefore, since 
$$\int\limits_{\mathbb{D}_{{\bf t}}} K({\bf t}-\eta)f_0(\eta) d \eta =\int\limits_{B(0,r)\cap \mathbb{D}_{{\bf t}}} K({\bf t}-\eta)f_0(\eta) d \eta +\int\limits_{B(0,r)^c \cap \mathbb{D}_{{\bf t}}} K({\bf t}-\eta)f_0(\eta) d \eta \, ,$$
we conclude
$$\Big{ |}\int\limits_{\mathbb{D}_{{\bf t}}} K({\bf t}-\eta)f_0(\eta) d \eta  \Big{ |}< \left( ||f_0||_{\infty} + \int\limits_{B(0,r)^c \cap \mathbb{D}_{{\bf t}}}| K({\bf t}-\eta)| d \eta \right) \epsilon\, .$$

\end{proof}

\begin{example}
In $\mathbb{R}^2$, let us consider the set $\mathbb{D}$ formed by the union of lines passing throughout a fixed point $p\in \mathbb{R}^2$, then we have 
$$\int_{\mathbb{D}_{\bf t}}K({\bf t}-\eta)f(\eta)d\eta  =0\, ,$$
for every ${\bf t} \in \mathbb{R}^2$. More generally, if $\mathbb{D}$ consists of sets contained in euclidean spaces of dimension less than $n$,  and after the canonical embedding of it into $\mathbb{R}^n$, we have that
$$\int_{\mathbb{D}_{\bf t}}K({\bf t}-\eta)f(\eta)d\eta  =0\, ,$$
for every ${\bf t} \in \mathbb{R}^n$. This clarifies the necessity of the condition $int(\mathbb{D}_{\bf t_0})\not = \emptyset$  for some ${\bf t_0} \in  \mathbb{R}^n$ (in previous Theorem), in order to have no nontrivial functions in the integral representation of $\Gamma f$.
\end{example}

\begin{example}
In the special case in which $\mathbb{D}=[\alpha_1, +\infty)\times[\alpha_2, +\infty)\times \cdots \times [\alpha_n, +\infty)$ we have $\mathbb{D}_{{\bf t}}=[\alpha_1, t_1]\times[\alpha_2, t_2 ]\times \cdots \times [\alpha_n, t_n]:=[{\bf \alpha},{\bf t}]$ and under the hypothesis (E1), (E2) and (E3) and assuming the notation 
$$\int_{{\bf \alpha}}^{{\bf t}}:=\int_{\alpha_1}^{t_1} \cdots \int_{\alpha_n}^{t_n}\, ,$$
we have that the integral function
$$\Gamma f({\bf t})=\int_{{\bf \alpha}}^{{\bf t}} K({\bf t}-\eta)f(\eta) d \eta\, ,$$
is $\mathbb{D}$- asymptotically ${\mathrm{R}}$ - multi-almost automorphic function, provided that $f$ is.
\end{example}

\begin{example}
Let $\alpha, \beta$ be positive real numbers and consider the kernel function $K_e(x-s,y-t)=\exp(-\alpha(x-s)) \exp(-\beta(y-t))$ and ${\bf t}=(x,y)$. 

Nevertheless if $\mathbb{D}$ is the first  quadrant $[0,+\infty)\times [0,+\infty)$ and $f$ an asymptotically $\mathbb{D}$-multi-almost automorphic function, then
$$F({\bf t})=\iint\limits_{\mathbb{D}_{{\bf t}}} K_e(x-s,y-t) f(s,t)dsdt\, ,$$
is not  $\mathbb{D}$- asymptotically ${\mathrm{R}}$ - multi-almost automorphic. In fact, let us consider the asymptotically $\mathbb{D}$-multi-almost automorphic function $f({\bf t})=1+e^{-(\alpha s+ \beta t)}$ where ${\bf t}=(s,t)$; then
\begin{eqnarray*}
F({\bf t})&=&\iint\limits_{\mathbb{D}_{{\bf t}}} K_e(x-s,y-t) (1+e^{-(\alpha s+ \beta t)})dsdt\\
&=&\iint\limits_{\mathbb{D}_{{\bf t}}} K_e(x-s,y-t)dsdt +\iint\limits_{\mathbb{D}_{{\bf t}}}K_e(x-s,y-t) e^{-(\alpha s+ \beta t)}dsdt\\
&=& \iint\limits_{\mathcal{I}_{{\bf t}}} K_e(x-s,y-t)dsdt -\iint\limits_{\mathcal{I}_{{\bf t}}\cap \mathbb{D}^{c}} K_e(x-s,y-t)dsdt +\\
&+& \iint\limits_{\mathbb{D}_{{\bf t}}}K_e(x-s,y-t) e^{-(\alpha s+ \beta t)}dsdt\\
&=& F_1({\bf t}) + F_2({\bf t}).
\end{eqnarray*}
where, 
$$F_1({\bf t}):=\iint\limits_{\mathcal{I}_{{\bf t}}} K_e(x-s,y-t)dsdt\,  ,$$
and 
$$F_2({\bf t}):=  \iint\limits_{\mathbb{D}_{{\bf t}}}K_e(x-s,y-t) e^{-(\alpha s+ \beta t)}dsdt-\iint\limits_{\mathcal{I}_{{\bf t}}\cap \mathbb{D}^{c}} K_e(x-s,y-t)dsdt\, .$$
We see that $F_1$ is $\mathrm{R}$-multi-almost automorphic (note that $K_e$ satisfy condition (E1)). On the other hand $F_2$ is not $\mathrm{R}$-multi-almost automorphic and for $(x_0,y) \in \mathbb{D}$ with fixed $x_0 \in (0,+\infty)$, we have
$$\lim_{|(x_0,y)|\to +\infty}F_2(x_0,y)\not =0\, .$$
\end{example}

Now let us prove a convolution type result for compact $\mathrm{R}$-multi-almost automorphic function.
 Let us first introduce the following definition of $\mathrm{R}$-multi-bi almost autmorphy:

\begin{defn}\label{defBaa} A jointly continuous function $G:\mathbb{R}^n\times \mathbb{R}^n\times X \to Y$ is $(\mathrm{R},\mathcal{B})$-multi-bi-almost automorphic if for any $B \in \mathcal{B}$ and any sequence $\{s_n\} \subset \mathrm{R}$, there exist a subsequence $\{s'_n\}\subseteq \{s_n\}$ and a function $G^*:\mathbb{R}^n\times \mathbb{R}^n\times X \to Y$ such that
\begin{equation}\label{EqNew01}
\lim_{n \to +\infty} G({\bf t}+s_n, {\bf s}+s_n,x)=G^*({\bf t}, {\bf s},x)\, 
\end{equation}
and 
\begin{equation}\label{EqNew02}
\lim_{n \to +\infty}G^*({\bf t}-s_n, {\bf s}-s_n,x)=G({\bf t}, {\bf s},x)\, ,
\end{equation}
pointwise for $({\bf t}, {\bf s}) \in \mathbb{R}^n \times  \mathbb{R}^n$ and any $x \in B$. 
If the limits in (\ref{EqNew01}) and (\ref{EqNew02}) are uniform on compact sets of $\mathbb{R}^n \times \mathbb{R}^n$, then we say that $G$ is compactly $(\mathrm{R},\mathcal{B})$-multi-bi-almost automorphic.
\end{defn}

\begin{rem} Let us note that, this generalization of the notion of Bi-almost automorphicity is different from the one given by the notion of $\mathrm{R}$-multi almost automorphicity, in which 
$${\mathrm R}=\{b : {\mathbb N} \rightarrow {\mathbb R}^{n} \, ; \, \mbox{ for all }j\in {\mathbb N}\mbox{ we have }b_{j}\in \{ (a,a,a,\cdot \cdot \cdot, a) \in {\mathbb R}^{n} : a\in {\mathbb R}\}\}.$$
\end{rem}

Next, we introduce some abstract results about the existence of compactly  $({\mathrm R},{\mathcal B})$-multi-almost automorphic (resp. asymptotically $({\mathrm R},{\mathcal B})$-multi-almost automorphic) solutions. 

\begin{thm}\label{TeoComb}
Let $Z$ be a Banach space and $G:\mathbb{R}^l \times \mathbb{R}^l \times X \to Z $ be compact $(\mathrm{R},\mathcal{B})$-multi-bi-almost automorphic such that there exists $ \lambda : \mathbb{R}^l \longrightarrow [0, +\infty) $  satisfying
\begin{equation}
 \| G({\bf t},{\bf s};x)-G({\bf t},{\bf s};y)\|_{Z} \leq \lambda ({\bf t}-{\bf s}) \|x-y\| \quad \text{ for all } x,y \in X,
\end{equation}
with $\sup_{ {\bf t} \in \mathbb{R}^{l}} \int_{\mathbb{R}^l}  \lambda ({\bf t}-{\bf s}) d{\bf s} < \infty$. 
Then, the operator $\Gamma $ defined by
$$\Gamma u({\bf t})= \int_{\mathbb{R}^l} G({\bf t},{\bf s};u({\bf s})) d{\bf s}\, ,$$
leaves invariant the space of $\mathrm{R}$-compact multi almost automorphic functions.
\end{thm}
\begin{proof}
Let $ u $ be compact $(\mathrm{R},\mathcal{B})$-multi-almost automorphic. Since the function $G$ is compact $\mathrm{R}$-multi-bi-almost automorphic, we can ensure that for given any sequence $\{\bf s'_n\}\subset \mathrm{R}$ there exist a subsequence $\{\bf s_n\}\subseteq \{\bf s'_n\}$ and a continuous function $G^{\ast}:\mathbb{R}^l \times \mathbb{R}^l \times X \to Z$ such that on compact subsets of $\mathbb{R}^l\times \mathbb{R}^l$ the following uniform limits hold
$$\lim\limits_{n\to +\infty} G({\bf t}+ {\bf s_n},{\bf s}+{\bf s_n };x)=:G^{\ast}({\bf t},{\bf s} ;x),\quad \lim\limits_{n\to +\infty} G^{\ast}({\bf t}-{\bf  s_n },{\bf s-s_n};x)=G({\bf t},{\bf s};x),$$
uniformly for $x \in B$. Moreover, there exists a function $\tilde{u}:\mathbb{R}^{l}\longrightarrow X $ such that on compact subsets of $\mathbb{R}^l$ the following uniform limits also hold
$$\lim\limits_{n\to +\infty} u({\bf t}+ {\bf s_n})=:\tilde{u}({\bf t}),\quad \lim\limits_{n\to +\infty} \tilde{u}({\bf t}- {\bf s_n})=u({\bf t}).$$

Let $v({\bf t}):=\Gamma u({\bf t})$ and $\tilde{v}({\bf t}):=\displaystyle \int_{\mathbb{R}^l} G^{\ast}({\bf t},{\bf s};\tilde{u}({\bf s}))d{\bf s}$.  Then,  it is clear by the Lebesgue's dominate convergence theorem that

\begin{align*}
&\| v({\bf t}+{\bf s_n})-\tilde{v}({\bf t}) \|= \\
&\| \int_{\mathbb{R}^l}  G({\bf t}+ {\bf s_n},{\bf s}+{\bf s_n };u({\bf s}+{\bf s_n })) d{\bf s} - \int_{\mathbb{R}^l} G^{\ast}({\bf t},{\bf s};\tilde{u}({\bf s}))d{\bf s}\| \\
& \leq  \| \int_{\mathbb{R}^l}  G({\bf t}+ {\bf s_n},{\bf s}+{\bf s_n };u({\bf s}+{\bf s_n })) d {\bf s} -  \int_{\mathbb{R}^l}  G({\bf t}+ {\bf s_n},{\bf s}+{\bf s_n };\tilde{u}({\bf s})) d {\bf s}  \|  \\ &+
\| \int_{\mathbb{R}^l}  G({\bf t}+ {\bf s_n},{\bf s}+{\bf s_n };\tilde{u}({\bf s})) d {\bf s} - \int_{\mathbb{R}^l} G^{\ast}({\bf t},{\bf s};\tilde{u}({\bf s}))d {\bf s} \|
\\
& \leq   \int_{\mathbb{R}^l} \lambda ({\bf t}-{\bf s}) \| u({\bf s}+{\bf s_n })) -\tilde{u}({\bf s}))  \| d {\bf s} \\ &+
 \int_{\mathbb{R}^l} \sup_{x\in B}\|  G({\bf t}+ {\bf s_n},{\bf s}+{\bf s_n };x) -G^{\ast}({\bf t},{\bf s};x) \| d {\bf s}.
\end{align*}
What we have proved is that $v=\Gamma u$ is an $(\mathrm{R}, \mathcal{B})$-multi almost automorphic function. Henceforth, in view of Theorem \eqref{TeoComb}, we prove that $v$ is $(\mathrm{R}, \mathcal{B})$-uniformly continuous. Let any two sequences $({\bf a}_{k}=(a_{k}^{1},a_{k}^{2},\cdots ,a_{k}^{n})) , ({\bf b}_{k}=(b_{k}^{1},b_{k}^{2},\cdot \cdot\cdot ,b_{k}^{n})) \in {\mathrm R}$ such that $ |{\bf a}_{k}-{\bf b}_{k} | \rightarrow 0$. Then, we have
\begin{align}
\sup_{x\in B}\| G \bigl({\bf a}_{k},{\bf a}_{k};x\bigr) - G\bigl({\bf b}_{k},{\bf b}_{k};x\bigr) \|_{Y} \rightarrow 0 \quad \text{ as } k\rightarrow +\infty. \label{Bi-Unifom assumption 1}
\end{align}
It is clear that  \eqref{Bi-Unifom assumption 1} implies that for all $ {\bf t}, {\bf s} \in \mathbb{R}^{l}$, we have
\begin{align}
\sup_{x\in B}\| G \bigl( {\bf t}+{\bf a}_{k}, {\bf s} +{\bf a}_{k};x\bigr) - G\bigl( {\bf t}+ {\bf b}_{k}, {\bf s} + {\bf b}_{k};x\bigr) \|_{Y} \rightarrow 0 \quad \text{ as } k\rightarrow +\infty. \label{Bi-Unifom assumption 2}
\end{align}
Furthermore, we have
\begin{align}
\| u \bigl({\bf a}_{k} \bigr) - u\bigl({\bf b}_{k}\bigr) \|_{Y} \rightarrow 0,
\end{align}
as $ k\rightarrow +\infty $. Then, 

\begin{align*}
&\| v \bigl({\bf a}_{k}\bigr) - v\bigl({\bf b}_{k}\bigr) \| =\\
&\|  \int_{\mathbb{R}^l}   G({\bf a}_{k},s;u({\bf s})) d {\bf s}-\int_{\mathbb{R}^l} G({\bf b	}_{k},{\bf s};u({\bf s})) d {\bf s} \| \\
& \leq  \|  \int_{\mathbb{R}^l}   G({\bf a}_{k}, {\bf s}+{\bf a}_{k} ;u({\bf s}+{\bf a}_{k})) d {\bf s}-\int_{\mathbb{R}^l} G({\bf b}_{k}, {\bf s}+{\bf b}_{k};u({\bf s}+{\bf b}_{k})) d {\bf s} \| \\
 &\leq   \|  \int_{\mathbb{R}^l}   G({\bf a}_{k},{\bf s}+{\bf a}_{k} ;u({\bf s}+{\bf a}_{k})) d {\bf s} -\int_{\mathbb{R}^l} G({\bf a}_{k},{\bf s}+{\bf a}_{k};u({\bf s}+{\bf b}_{k})) d {\bf s} \| \\  & +
 \|  \int_{\mathbb{R}^l}   G({\bf a}_{k}, {\bf s}+{\bf a}_{k} ;u({\bf s}+{\bf b}_{k})) d{\bf s}-\int_{\mathbb{R}^l} G({\bf b}_{k},{\bf s}+{\bf b}_{k};u({\bf s}+{\bf b}_{k})) d {\bf s} \| \\
 &\leq   \int_{\mathbb{R}^l}  \lambda({\bf s}) \| u({\bf s}+{\bf a}_{k})- u({\bf s}+{\bf b}_{k}) \| d{\bf s} \\  & +
  \int_{\mathbb{R}^l}  \sup_{x\in B} \| G({\bf a}_{k},{\bf s}+{\bf a}_{k} ;x) -G({\bf b}_{k},{\bf s}+{\bf b}_{k};x) \| d {\bf s}.
\end{align*}
\end{proof}
 From the proof of Theorem  \ref{TeoComb}, one may deduce that the following result holds immediately.
\begin{cor}\label{CoroComb}
Let $Z$ be a Banach space and $G:\mathbb{R}^l \times \mathbb{R}^l \times X \to Z $ be compact $(\mathrm{R},\mathcal{B})$-multi-bi-almost automorphic such that there exists $ \lambda : \mathbb{R}^l \longrightarrow [0, +\infty) $  satisfying
\begin{equation}
 \| G({\bf t},{\bf s};x)-G({\bf t},{\bf s};y)\|_{Z} \leq \lambda ({\bf t}-{\bf s}) \|x-y\| \quad \text{ for all } x,y \in X, \label{Formula (3.6)}
\end{equation}
with $\sup_{ {\bf t} \in \mathbb{R}^{l}} \int_{\mathbb{R}^l}  \lambda ({\bf t}-{\bf s}) d{\bf s} < \infty$. 
Then, the operator $\Gamma $ defined by
$$\Gamma u({\bf t})= \int_{\mathbb{R}^l} G({\bf t},{\bf s};u({\bf s})) d{\bf s}\, ,$$
leaves invariant the space of $\mathrm{R}$-compact multi almost automorphic functions.
\end{cor}
\begin{thm}\label{Exist comp multi-AA nonlinear}
Let $Z$ be a Banach space and $G:\mathbb{R}^l \times \mathbb{R}^l \times X \to Z $ be compact $(\mathrm{R},\mathcal{B})$-multi-bi-almost automorphic such that there exists $ \lambda : \mathbb{R}^l \longrightarrow [0, +\infty) $  satisfying
\begin{equation}
 \| G({\bf t},{\bf s};x)-G({\bf t},{\bf s};y)\|_{Z} \leq \lambda ({\bf t}-{\bf s}) \|x-y\| \quad \text{ for all } x,y \in X,
\end{equation}
with $\sup_{ {\bf t} \in \mathbb{R}^{l}} \int_{\mathbb{R}^l}  \lambda ({\bf t}-{\bf s}) d{\bf s} < \infty$. Assume that $\sup_{ {\bf t} \in \mathbb{R}^{l}} \int_{\mathbb{R}^l}  \lambda ({\bf t}-{\bf s}) d{\bf s} <1$. 
Then, $\Gamma$ has a unique compact $(\mathrm{R},\mathcal{B})$-multi-almost automorphic solution defined by
$$ u({\bf t})=  \int_{\mathbb{R}^l} G({\bf t},{\bf s};u({\bf s})) d{\bf s}, \quad  {\bf t} \in \mathbb{R}^{l} .$$
\end{thm}
\begin{proof}
Let $ u,v $ be compact $(\mathrm{R},\mathcal{B})$-multi-almost automorphic functions. Then,
\begin{eqnarray*}
\| \Gamma u({\bf t})- \Gamma v({\bf t}) \|_{Y} &\leq & \int_{\mathbb{R}^{l}} \| G({\bf t},{\bf s};u({\bf s}))-G({\bf t},{\bf s};v({\bf s}))\|_{Z} \ d {\bf s} \\ 
&\leq & \int_{\mathbb{R}^{l}} \lambda ({\bf t}-{\bf s})  \| u({\bf s})- v({\bf s})\|_{Y} \ d {\bf s} \\
&\leq & \int_{\mathbb{R}^{l}} \lambda ({\bf t}-{\bf s})  \ d {\bf s}  \| u- v\|_{\infty}, \quad  {\bf t} \in \mathbb{R}^{l}.
\end{eqnarray*}
Hence, $$  \| \Gamma u- \Gamma v \|_{\infty}  \leq  \sup_{ {\bf t} \in \mathbb{R}^{l}}  \int_{\mathbb{R}^{l}} \lambda ({\bf t}-{\bf s})  \ d {\bf s}  \| u- v\|_{\infty}.$$
So, the result follows in view of the Banach contraction principle.
\end{proof}
Similarly to Theorem \ref{Exist comp multi-AA nonlinear}, we prove the existence and uniqueness results of $\mathbb{D}$-asymptotically ${\mathrm{R}}$-multi-almost automorphic solutions to equation. Moreover, we have the following existence result.
\begin{itemize}
\item [(E4)] Let $ F:\mathbb{R}^{n}\times X\longrightarrow Y $ and $L_F \geq 0 $ such that 
$$  \| F( {\bf t};x)-F( {\bf t};y) \|_{Y} \leq L_F \|x-y \|, \quad {\bf t} \in \mathbb{R}^{l}, \, x,y \in X. $$
\end{itemize}
\begin{prop}
Let $ F, \ G:\mathbb{R}^{n}\times X\longrightarrow Y $ be $\mathbb{D}$-asymptotically $({\mathrm{R}},\mathcal{B})$-multi-almost automorphic such that the conditions $(E1)$-$(E4)$ (for $F$ and $G$ with $L_G$-Lipschitz constant) are  satisfied, and let $\mathbb{D}_{{\bf t}}:=\mathcal{I}_{{\bf t}}\cap \mathbb{D}$ with $int(\mathbb{D}_{{\bf t_0}})\not = \emptyset$ for some ${\bf t_0}\in \mathbb{R}^n$. Then, the integral operator
$$  \Gamma u({\bf t})=G({\bf t};u({\bf t})) +\int\limits_{\mathbb{D}_{{\bf t}}} K({\bf t}-\eta)F(\eta, u(\eta))d \eta \, ,$$
maps the space of $\mathbb{D}$-asymptotically ${\mathrm{R}}$-multi-almost automorphic functions into itself.
\end{prop}
\begin{proof}
Let $ u $ be $\mathbb{D}$-asymptotically ${\mathrm{R}}$-multi-almost automorphic. Then, Theorem \ref{eovakoonako} yields that $\eta \longmapsto F(\eta,u(\eta)), \ G(\eta,u(\eta))$ are $\mathbb{D}$-asymptotically ${\mathrm{R}}$-multi-almost automorphic functions. Hence, from Theorem \ref{Theorem existence linear VIO}, we deduce that $ {\bf t} \longmapsto \Gamma u({\bf t}) $ as sum of two $\mathbb{D}$-asymptotically ${\mathrm{R}}$-multi-almost automorphic functions.
\end{proof}
\begin{thm}\label{Asym Nonli integral Volt}
Let $ F, \ G:\mathbb{R}^{n}\times X\longrightarrow Y $ be $\mathbb{D}$-asymptotically $({\mathrm{R}},\mathcal{B})$-multi-almost automorphic such that the conditions $(E1)$-$(E4)$ are all satisfied, and let $\mathbb{D}_{{\bf t}}:=\mathcal{I}_{{\bf t}}\cap \mathbb{D}$ with $int(\mathbb{D}_{{\bf t_0}})\not = \emptyset$ for some ${\bf t_0}\in \mathbb{R}^n$. Assume that $L_G + L_F  \sup_{{\bf t} \in \mathbb{R}^{l}} \int\limits_{\mathbb{D}_{{\bf t}}} K({\bf t}-\eta)  <1$. Then, the integral operator
$$  \Gamma u({\bf t})=G({\bf t};u({\bf t})) + \int\limits_{\mathbb{D}_{{\bf t}}} K({\bf t}-\eta)F(\eta, u(\eta))d \eta \, ,$$
has a unique $\mathbb{D}$-asymptotically ${\mathrm{R}}$-multi-almost automorphic solution i.e., $\Gamma u=u$.
\end{thm}
\begin{proof}
Let $ u,v $ be $\mathbb{D}$-asymptotically ${\mathrm{R}}$-multi-almost automorphic. Then,
\begin{eqnarray*}
\| \Gamma u({\bf t})- \Gamma v({\bf t}) \|_{Y} &\leq & \| G({\bf t};u({\bf t}))-G({\bf t};v({\bf t})) \|+ \int\limits_{\mathbb{D}_{{\bf t}}} K({\bf t}-\eta) \| F(\eta, u(\eta))-F(\eta, v(\eta))\|  d \eta\\ 
&\leq & L_G \| u({\bf t})-v({\bf t}) \| + L_F \int\limits_{\mathbb{D}_{{\bf t}}} K({\bf t}-\eta) \| u({\eta})- v({\eta})\|_{Y} \ d \eta  \\
&\leq & \left( L_G+L_F  \sup_{{\bf t} \in \mathbb{R}^{l}}  \int\limits_{\mathbb{D}_{{\bf t}}} K({\bf t}-\eta) \ d \eta \right)   \| u- v\|_{\infty}, \quad  {\bf t} \in \mathbb{R}^{l}.
\end{eqnarray*}
Hence, $$  \| \Gamma u- \Gamma v \|_{\infty}  \leq  \left( L_G+L_F  \sup_{{\bf t} \in \mathbb{R}^{l}}  \int\limits_{\mathbb{D}_{{\bf t}}} K({\bf t}-\eta) \ d \eta \right)   \| u- v\|_{\infty}.$$
So, the result follows in view of the Banach contraction principle.
\end{proof}

\section[Applications]{Applications to the abstract multidimensional Volterra integral equations and partial differential equations}\label{manuel}
In this section, we present different applications to a wide class of nonlinear multi-dimensional Volterra integral equations, linear  and nonlinear partial differential equations, and integrodifferential equations. 

\subsection{Almost automorphic solutions to nonlinear multidimensional Volterra integral equations with infinite delay}\label{Nn AVIEDinf}
Let ${\bf t}=(x,y)$, we aim to study the following integral equation of Volterra type with infinite delay:
\begin{equation}\label{VIE01}
f({\bf t})=g({\bf t})+\int_{-\infty}^x\int_{-\infty}^y K({\bf t}-\eta)h(\eta,f(\eta))d\eta
\end{equation}
where $g:\mathbb{R}^2 \to Y$ is compact $\mathrm{R}$-multi-almost automorphic and  $h:\mathbb{R}^2\times X\to Y$ is compactly $(\mathrm{R},\mathcal{B})$-multi almost automorphic, and $\mathcal{B}$ consists of bounded subsets of $X$. Assume that $h$ is $L_h$-Lipschitzian in the second variable. Define $G({\bf t}, \eta ;x) :=K({\bf t}-\eta )\chi_{\mathcal{I}_{{\bf t}}}(\eta)  h(\eta,x) $. By construction, the map $G$ is compactly $(\mathrm{R},\mathcal{B})$-multi-bi-almost automorphic and it satisfies \eqref{Formula (3.6)} with $\lambda(\cdot)= L_h \ K(\cdot )\chi_{\mathcal{I}_{{\bf t}}}(\cdot)  $ which is assumed to be integrable (if the kernel $K$ is integrable) such that $ L_h \sup_{{\bf t }} \int_{-\infty}^x\int_{-\infty}^y K( {\bf t}-\eta ) d \eta <1$. Therefore, since $g$ is compactly $\mathrm{R}$-multi-almost automorphic, it follows in view of Theorem \ref{Exist comp multi-AA nonlinear}, that the integral operator in the equation (\ref{VIE01}) has a unique compact $\mathrm{R}$-multi-almost automorphic solution.
\begin{example}
Consider the following  application to the two-dimensional nonlinear Volterra integral equation of the second kind with infinite delay, see \cite{AD,AIK} for some examples in the absence of delay (i.e., $x,y \geq 0$), namely
\begin{equation}
u(x, y) = g(x, y) + \int_{-\infty}^x \int_{-\infty}^y K(x, y, s, t, u(s, t)) ds dt, \quad  ( x, y) \in \mathbb{R}^2 , \label{VIE01 Example}
\end{equation}
where $u$ is the (unknown) state function called the solution of two-dimensional integral Volterra equation, the kernel $K$ is a known nonlinear function in $u$ and $g$ is also a known function. This kind of integral equations have significant applications in various domains of applied science and engineering, see the references mentioned in the papers \cite{AD,AIK} about the modeling overview of this kind of integral equations. Moreover, this kind of multi-dimensional integral operators appears a solution of nonlinear elliptic differential equations, see \cite[Chapter 10]{Cor-Hil}. Let us consider
\begin{eqnarray*}
K(x, y, s, t, u(s, t)) &:=&\gamma e^{-|x-s |-| y-t|}\cos(s)\sin(t)\ln(1+\mid u(s,t) \mid )\\
 &=& k(s-x,t-y)h((s,t),u(s,t))\ln(1+\mid u(s,t) \mid ),
\end{eqnarray*}
$k(\eta_1,\eta_2)=e^{-|\eta_1|-| \eta_2|}$ is the kernel and $h(\eta_1,\eta_2,u(\eta_1,\eta_2))=\gamma \cos(\eta_1)\sin(\eta_2) \ln(1+\mid u(\eta_1,\eta_2) \mid )$  which is $\gamma$-Lipschitzian in the second variable and $\gamma$ is chosen sufficiently small.  Moreover, we take $$g(x,y)=\left[ \sin(x)+\sin(\pi x)\right] \left[ \cos(x)+\cos(\pi x)\right]+ \dfrac{1}{\sqrt{1+x^{2}+y^{2}}}.$$ Hence, under the above considerations, equation  \eqref{VIE01 Example} has the same form as  \eqref{VIE01}. Further on, all the assumptions and hypotheses of Section \ref{Nn AVIEDinf} are satisfied. Consequently, equation \eqref{VIE01 Example} has a unique compact almost automorphic  solution.
\end{example}

\subsection{Asymptotically almost automorphic solutions to abstract Volterra integral equations}
Let ${\bf t}=(x,y)$, $\mathbb{D}=[0,+\infty)\times [0,+\infty)$ and $\mathbb{D}_{{\bf t}}=[0,x]\times [0,y]$. Here we prove the existence and uniqueness of $\mathbb{D}$-asymptotically ${\mathrm{R}}$-multi-almost automorphic solutions to the following abstract integral equation of Volterra type, namely:
\begin{equation}\label{VIE02}
f({\bf t})=g({\bf t};f({\bf t}))+\int_0^x\int_0^y K({\bf t}-\eta)h(\eta,f(\eta))d\eta,
\end{equation}
where $g:\mathbb{R}^2 \times X \to Y$  and  $h:\mathbb{R}^2\times X\to Y$ are two $\mathbb{D}$-asymptotically $({\mathrm{R}},\mathcal{B})$-multi-almost automorphic functions with respect to ${\bf t}$ and respectively $L_g$, $L_h$ Lipschitzian with respect to $x$, whereas $\mathcal{B}$ consists of the collection of bounded subsets of $X$. For the kernel $K$ it is assumed that it satisfies conditions $(E1)$, $(E2)$ and $(E3)$. Hence, under the assumption that $L_g + L_h  \sup_{{\bf t} \in \mathbb{R}^{l}}\int_0^x\int_0^y K({\bf t}-\eta) d \eta <1$ it yields from Theorem \ref{Asym Nonli integral Volt} that the integral equation (\ref{VIE02}) has a unique $\mathbb{D}$-asymptotically ${\mathrm{R}}$-multi-almost automorphic solution.
\begin{example}
Consider the following nonlinear wave  partial differential equation:
\begin{equation}\label{PDE app Volt 2D}
\dfrac{\partial^{2} u}{\partial t^{2}} =\dfrac{\partial^{2} u}{\partial x^{2}}  -\delta \cos(t)e^{-x} \sin(u) , \, t,x \geq 0,
\end{equation}
with the initial conditions given on the $x=\pm t$ i.e., $u(x,x)=e^{-x}$ and $u(t,-t)=-e^{-t}$. We argue similarly as in the paper \cite{McTaDio} and we use the transformation  around the lines $ y=x+t$ and $s=x-t$. Define the new state function by 
$$  v(y,s)=\dfrac{1}{4}u\left(\dfrac{1}{2}(Y+S),\dfrac{1}{2}(Y-S)\right) .$$
Then, a further calculation yields that equation \eqref{PDE app Volt 2D} is equivalent to the following integral two-dimensional Volterra nonlinear operators of the form
\begin{eqnarray*}
v(y, s) &=& g (y, s,u(y,s)) -\delta \int_{0}^y \int_{0}^s  \sin (v(Y,S)) \cos \left(\dfrac{1}{2}(Y-S)\right) e^{-\dfrac{1}{2}(Y+S)} dYdS,
\end{eqnarray*}
where $$ g (y, s,v(y,s)) =e^{-\frac{y}{2}}-e^{-\frac{s}{2}}. $$
Then the kernel $K(Y,S)=- e^{-\dfrac{1}{2}(Y+S)}$ satisfies the assumptions $(E1)$, $(E2)$ and $(E3)$. The function $g$ is $\mathbb{D}$-asymptotically-multi-almost automorphic and $h(y,s,v(y,s))=\delta \cos (\dfrac{1}{2}(Y-S))  \sin (v(Y,S))$, which is $\delta$-Lipschitzian with respect to $v$ and $\mathbb{D}$-asymptotically-multi-almost automorphic with respect to $(Y,S)$. Thus, by taking $\delta$ small enough, we obtain that the equations \eqref{PDE app Volt 2D} has a unique $\mathbb{D}$-asymptotically-multi-almost automorphic solution.
\end{example}


%



\subsection{Poisson equation: $\mathrm{R}$-Compact multi-almost automorphic solutions.}

Let us denote by $\Delta$ the Laplace operator and $\mathrm{R}$ be the set of sequences such that all its subsequences are also in $\mathrm{R}$. In this subsection we shall prove the following main Theorem which extend (and improve) Sibuya's result for multi-almost periodic functions (\cite{YSibuya}). 
\begin{thm}\label{TeoPoisson}
Let $f$ be an $\mathrm{R}$-multi-almost automorphic function and let $u \in C_b(\mathbb{R}^n)$. Assume that $u$ solves the (Poisson) equation i.e.,
\begin{equation}\label{PoissonEq}
\Delta u=f
\end{equation}
in the distributional sense. Then, $u$ is $\mathrm{R}$-compact almost automorphic.
\end{thm}

First, let us define what we mean by solution in the distributional sense. Let $C^{\infty}_{c}(\mathbb{R}^n)$ denote the space of test functions in $ \mathbb{R}^n $ i.e., functions of class $C^{\infty}$ with compact support in $ \mathbb{R}^n $.
\begin{defn}
A function $u:\mathbb{R}^n \to \mathbb{R}$ is a solution of equation (\ref{PoissonEq}) in the sense of distributional sense, if it satisfies
\begin{equation}
\int_{\mathbb{R}^n}u(x)\Delta \phi(x)dx = \int_{\mathbb{R}^n}f(x)\phi(x)dx\, \, , \,\,\, \forall \phi \in C^{\infty}_{c}(\mathbb{R}^n)\, . \label{distributuions}
\end{equation}
\end{defn}
For the proof, we follows Sibuya's paper \cite{YSibuya}. 


\begin{proof}
({\bf Proof of Theorem (\ref{TeoPoisson})})
Let $\{ \zeta_n\}$ be a sequence in $\mathrm{R}$. Since $f$ is $\mathrm{R}$-multi almost automorphic, there exist a subsequence $ \{ \omega_n\}\subset \{ \zeta_n\}$ and a measurable bounded function $g:\mathbb{R}^n \to \mathbb{R}$ such that the following limits hold:
\begin{equation}
\lim_{n \to + \infty}f(x+\omega_n)=g(x)\quad \text{ and }
\lim_{n \to + \infty}g(x-\omega_n)=f(x)\, 
\end{equation}
pointwise in $x\in \mathbb{R}^{n} $. Moreover, since $f$ is bounded, it follows in view of the formula \eqref{distributuions}, the sequence $ (u_n (\cdot):=u(\cdot+\omega_n))_n $ converges pointwise in $x\in \mathbb{R}^{n} $.  Furthermore, we point out that the uniform continuity of $u$ yields (see \cite{Cor-Hil} on $\mathbb{R}^n$ that the family $(u_n (\cdot))_n$ is equicontinuous. Therefore, we obtain the following uniform limit on compact subsets of $\mathbb{R}^n$, i.e., 
\begin{equation}
\lim_{n \to + \infty}u(x+\omega_n):=v(x) . \label{limit1 Sib}
\end{equation}
Note that the function $v$ is also uniformly continuous and bounded. Now, we prove that 
\begin{equation}
\lim_{n \to + \infty}v(x-\omega_n)=u(x) ,  \label{limit2 Sib}
\end{equation}
holds uniformly on compact subsets of $\mathbb{R}^n$.
Notice first that $v$ is solution of the equation
$$\Delta v=g\, ,$$
in the sense of distributions. Arguing as above and since $v$ is uniformly continuous, there exist a subsequence $ \{ \xi_n\} \subset \{ \omega_n\}$ and a bounded uniformly continuous functions $w$ such that the following limit holds uniformly on compact subsets of $\mathbb{R}^n$
\begin{equation}
\lim_{n \to + \infty}v(x-\omega_n)=w(x) . \label{limit3 Sib}
\end{equation}
Moreover $w$ satisfies the equation 
$$\Delta w=f,$$
in the sense of distributions. Therefore, the function $u-w$ is a bounded and harmonic. Hence, there exists a constant $C$ such that
$$u=w+C.$$
In order to prove that $C=0$, it suffices to see that 
$$\inf u \leq \inf v \leq \inf w$$
and
$$\sup u \geq \sup v \geq \sup w\, .$$
This is consequence of  \eqref{limit1 Sib}, \eqref{limit2 Sib} and \eqref{limit3 Sib} respectively, which proves the result.

\end{proof}


\subsection{heat equation: : $\mathrm{R}$-almost automorphic solutions.}

In this subsection we study the initial value problem for the homogeneous heat equation with nonlocal diffusion
\begin{eqnarray}\label{HEap}
u_t - \Delta  u&=& 0 \,\, \,\, {\rm in} \,\,\,\, \mathbb{R}^n \times (0,+\infty), \\
u(x,0^+)&=&g(x)\,\, \,\, {\rm in} \,\,\,\, \mathbb{R}^n \times \{0\},\nonumber
\end{eqnarray}

\begin{prop}
Let the initial profile $g$ of equation (\ref{HEap}) be a $\mathrm{R}$-multi almost automorphic function. Then its solution $u$ is $\mathrm{R}$-multi-almost automorphic in space variable.
\end{prop}
\begin{proof}
From \cite{evans}, we have that the representation formula of the solution for initial value problem (\ref{HEap}) is given by
\begin{equation}
u(x,t)=\int_{\mathbb{R}^n}\Phi(x-y,t)g(y)dy\, ,
\end{equation}
where $\Phi(x,t)$ is the fundamental solution which satisfies 
\begin{equation}\label{HE01}
\int_{\mathbb{R}^n}\Phi(\xi,t)d\xi=1\, .
\end{equation}
The rest of the proof is similar to that of Proposition \ref{convdiaggas}.
\end{proof}
Note that, if $f$ is $\mathrm{R}$-multi-almost periodic (compactly ${\mathrm R}$-multi-almost automorphic function), then its solution is also $\mathrm{R}$-multi-almost periodic (compactly ${\mathrm R}$-multi-almost automorphic function respectively).

\subsection{Heat conduction in materials with memory: $\mathrm{R}$-multi-almost automorphic solutions.}

In this subsection we suppose that $\mathrm{R}\subset \mathbb{R}$ is a subset such that conditions [D1] and [D2] holds, also we denote by $\mathrm{R}AAA(\mathbb{R}^+; X)$ and by $\mathrm{R}AAA(\mathbb{R}^+\times X; X)$ the Banach space of asymptotically $\mathrm{R}$-almost automorphic functions and asymptotically $(\mathrm{R},\mathcal{B})$-almost automorphic functions respectively (see Theorem \ref{2.1.10obazac}).

The following integrodifferential equation (with nonlocal initial data) models the heat conduction in materials with memory
\begin{eqnarray}
u'(t)&=& Au(t)+ \int_0^t B(t-s)u(s)ds+f(t,u(t)),\ t\geq0\; , \label{eqAX01}\\
u(0)&= & u_0 + g(u)\, ; \label{eqAXX01}
\end{eqnarray}
where $u_0 \in X$, $A$ and $\{B(t)\}_{ t \geq  0}$ are linear, closed and densely defined operators on the Banach space $X$.

Results about existence and uniqueness of the asymptotically almost automorphic solution to the integrodifferential equation (\ref{eqAX01})-(\ref{eqAXX01}) have been studied in \cite{chavez3,HSDing}. This subsection is an adaptation of the results on asymptotically almost automorphic solutions to integrodifferential equations in \cite{chavez3} to the case of asymptotically $\mathrm{R}$-almost automorphic ones, and we give details for the reader convenience. Let us start with the following definition

%

\begin{defn}\label{d14}\ {\rm A family \{$R(t)\}_{t\geq 0}$ of continuous linear operators on $X$ is called a resolvent
operator for the equation \ (\ref{eqAX01}) \ if, and only if:}
\begin{enumerate}
\item [{\rm (R1)}] {\rm $R(0)=I$, is the identity operator on $X$.}
\item [{\rm (R2)}] {\rm for all $x\in X$, the operator $t \rightarrow R(t)x$ is a continuous function on
$[0,+\infty)$.}
\item [{\rm (R3)}] {\rm For all $t\geq0$ the operator $R(t)$ is continuous on $Y$, and for all $y\in Y$, the application $t\rightarrow R(t)y$ belongs to $C([0,+\infty);Y)\cap C^1([0,+\infty);X)$  and satisfies}
$$\frac{d}{dt}R(t)y=AR(t)y+\int_0^tB(t-s)R(s)yds=R(t)Ay+\int_0^tR(t-s)B(s)yds ,\ t\geq0, $$
\end{enumerate}

\noindent {\rm where $Y=D(A)=B(t)$ for all $t\geq 0$ and is equipped with the graph norm. 
Details on resolvent operators and same conditions on their existence can be found in \cite{RCGrim,JPSS}.}
\end{defn}

\noindent We assume that the resolvent operator $\{R(t)\}_{t\geq0}$ of equation (\ref{eqAX01}) exists; under this situation, we can define a mild solution of equations (\ref{eqAX01})-(\ref{eqAXX01}) as follows:
\begin{defn}\label{d15}\ {\rm A function $u\in C(\mathbb{R}^+;X)$ is a mild solution of the nonlocal integrodifferential equation (\ref{eqAX01})-(\ref{eqAXX01}) if}
$$u(t)=R(t)\left( u_0+g(u) \right)+\int_0^tR(t-s)f(s,u(s))ds,\ t\geq 0 \, .$$
\end{defn}

\begin{defn}{\bf (Uniform exponential stability).} {\rm 
The resolvent operator  $\{R(t)\}_{t\geq0}$ has uniform exponential stability, if there exist positive constants
$M\geq 1, \, \delta>0$ such that for all $t\geq 0$ we have $\|R(t)\|\leq Me^{-\delta t}$.}
\end{defn}
\noindent We say that the resolvent operator $\{R(t)\}_{t\geq0}$  has property {\bf (R)}, if it has uniform exponential stability.
%

The proof of the following two local theorems are adaptations of the ones contained in \cite{chavez3}.

\begin{thm}\label{th31} Let $f \in \mathrm{R} AAA(\mathbb{R}^+\times X; X)$, $g: \mathrm{R} AAA(\mathbb{R}^+; X)\to X$, the resolvent operator has property {\bf (R)},
 $\rho >0$, and the set $\Delta_0=\{y \in \mathrm{R} AAA(\mathbb{R}^+,X):||y-y_0||_{\infty}\leq \rho\},$ where 
$$ y_0(t)=R(t)(u_0+g(0))+\int_{0}^{t}R(t-s)f(s,0)ds.$$
Suppose also that $||y_0||_{\infty}\leq \rho$, there exist positive constants $L_f , L_g$ such that:
$$||f(t,x(t))-f(t,y(t))||\leq L_f||x(t)-y(t)||,\ x,y \in \Delta_0, t \geq 0 ,$$
and $g$ is $L_g$-Lipschitz in $\Delta_0$. If the constants $L_g, L_f,\rho, M,\delta$ satisfy the inequality:
\begin{equation}\label{eq33}
\delta L_g+L_f < \dfrac{\rho \delta}{M(\rho+||y_0||_{\infty})}\, ;
\end{equation}
then, equation (\ref{eqAX01})-(\ref{eqAXX01}) has a unique mild solution in $\Delta_0$.
\end{thm}

Note that, under hypothesis of theorem \ref{th31}, we have the well defined operator $\Gamma: \mathrm{R} AAA(\mathbb{R}^+; X) \to \mathrm{R}AAA (\mathbb{R}^+; X)$, where
\begin{equation}\label{finop}
\Gamma u(t)=R(t)\left( u_0+g(u) \right)+\int_0^tR(t-s)f(s,u(s))ds.
\end{equation}

\begin{thm}\label{th313}
Let $f \in \mathrm{R} AAA(\mathbb{R}^+\times X; X)$, $g: \mathrm{R} AAA(\mathbb{R}^+; X)\to X$, the resolvent operator satisfies property {\bf (R)}, $\rho>0$ and the set 
$\Delta_0=\{y\in \mathrm{R} AAA(\mathbb{R}^+; X): ||y-y_0||_{\infty}\leq\rho\},$ where 
 $$y_0(t)=R(t)(u_0+g(0))+\int_0^t R(t-s)f(s,0)ds\, .$$
Suppose, further that there exist positive constants $L_f, L_g$ such that:
$$\|f(t,x(t))-f(t,y(t))\|\leq L_f \|x(t)-y(t) \|,\ x,y \in \Delta_0, \ t\geq 0,$$
and $g$ is $L_g$-Lipschitz in $\Delta_0$. If $\Gamma$ is the operator defined in (\ref{finop}) and the inequality 
$$0<\theta=\left( 1-ML_g-\dfrac{M}{\delta}L_f \right) ^{-1}||\Gamma y_0-y_0||_{\infty}\leq\rho\, $$
holds. Then equation (\ref{eqAX01})-(\ref{eqAXX01}) has a unique mild solution  $y\in \Delta_0$. 
\end{thm}

After introducing this two local abstract theorems for the integrodifferential equation (\ref{eqAX01})-(\ref{eqAXX01}), let us apply the results to the heat conduction in materials with memory.



Let $\Omega \subset \mathbb{R}^3$ be any open bounded set with smooth boundary $\partial\Omega$. The heat conduction in materials with memory is described with the following partial integrodifferential equation:

\begin{eqnarray}\label{EQ1}
\frac{\partial^2 \theta}{\partial^2 t}(x,t)+\beta(0)\frac{\partial \theta}{\partial
t}(x,t)&=&\alpha(0)\Delta\theta (x,t)-\int_{-\infty}^{t}\beta'(t-s)\frac{\partial \theta}{\partial
t}(x,s)ds\\
&+&\int_{-\infty}^{t}\alpha'(t-s)\Delta\theta(x,s)ds+a(t)b(\theta(t)),\nonumber
\end{eqnarray}
where $\alpha, \beta \in C^2([0,+\infty[;\mathbb{R})$, are the thermal relaxation function of the heat flux and the energy
relaxation function respectively, with $\alpha (0),\beta (0)$ positive and $\Delta$ is the Laplace operator in $\Omega$. For isotropic  materials, in which the temperature $\theta(x,t)$ is know for all  $t\leq 0$ and does not depend on the position  $x\in \Omega$, equation (\ref{EQ1}) can be rewritten as follows:

\begin{eqnarray}\label{EQ2}
\theta^{''}(t)+\beta(0)\theta^{'}(t)&=&\alpha(0)\Delta\theta (t)-\int_{0}^{t}\beta'(t-s)\theta^{'}(s)ds+\\
&+&\int_{0}^{t}\alpha'(t-s)\Delta\theta(s)ds+a(t)b(\theta(t)).\nonumber
\end{eqnarray}
Introducing the function $\eta(t)=\theta'(t)$, the previous equation is written as the system:
\begin{eqnarray*}
 \theta'(t) &=& \eta(t).\\
\eta'(t)&=&-\beta(0)\eta(t)+\alpha(0)\Delta\theta (t)-\int_{0}^{t}\beta'(t-s)\eta(s)ds+
\int_{0}^{t}\alpha'(t-s)\Delta\theta(s)ds+a(t)b(\theta(t))\, ,
\end{eqnarray*}
that is :

\begin{eqnarray*}
\begin{bmatrix}
 {\theta(t)}\\
{\eta(t)}
\end{bmatrix}^{'}
=\begin{bmatrix}
{0}&{I}\\
{\alpha (0)\Delta}&{-\beta(0)I}
\end{bmatrix}
\begin{bmatrix}
 {\theta(t)}\\
{\eta(t)}
\end{bmatrix}+
\int_{0}^{t}\begin{bmatrix}
{0}&{0}\\
{\alpha'(t-s)\Delta}&{-\beta'(t-s)I}
\end{bmatrix}\begin{bmatrix}
 {\theta(s)}\\
{\eta(s)}
\end{bmatrix}ds+\begin{bmatrix}
{0}\\
{a(t)b(\theta(t))}
\end{bmatrix}.
\end{eqnarray*}
Now, consider the ambient space $X=H_0^1(\Omega)\times L^2(\Omega)$ and identify the operators:
\begin{eqnarray*}
 A=\begin{bmatrix}
{0}&{I}\\
{\alpha (0)\Delta}&{-\beta(0)I}
\end{bmatrix} , \,  B(t-s)=\begin{bmatrix}
{0}&{0}\\
{\alpha'(t-s)\Delta}&{-\beta'(t-s)I}
\end{bmatrix}.
\end{eqnarray*}
Furthermore for $\begin{bmatrix}
{\theta}\\
{\eta}
\end{bmatrix}$ $\in X$ identify:

\begin{eqnarray*}
 u=\begin{bmatrix}
{\theta}\\
{\eta}
\end{bmatrix} , \, \, f(t,u)=\begin{bmatrix}
{0}\\
{a(t)b(\theta)}
\end{bmatrix}\, .
\end{eqnarray*}

Then, for each initial condition 
$u(0)=u_0=\begin{bmatrix}
  {\theta_0}\\{\eta_0}
 \end{bmatrix} \in X
$ 
%
%
we have the following integrodifferential equation with nonlocal initial condition:
\begin{eqnarray}\label{EQ3}
  u'(t)&=&Au(t)+\int_0^t B(t-s)u(s)ds+f(t,u(t)), \ t\geq 0,\\
 u(0)&=&u_0+h(u),\label{EQ33}
\end{eqnarray}
with $D(A)=(H^2(\Omega)\cap H_0^1(\Omega))\times H_0^1(\Omega)$.

It follows from \cite{GChen} (see also \cite{RCGrim}), that $A$ generates a semigroup $\{T(t)\}_{t\geq 0}$ with
$||T(t)||\leq Me^{-\gamma t}$ for all $t\geq 0$ and $M,\gamma$ positive constants . Let $B(t)=F(t)A$, where:

\begin{eqnarray*}
 F(t)=\begin{bmatrix}
{0}&{0}\\
{-\beta'(t)I+\beta(0)\frac{\alpha'(t)}{\alpha(0)}I}&{\frac{\alpha'(t)}{\alpha(0)}I} 
      \end{bmatrix}.
\end{eqnarray*}

It is well known that under the following conditions: \\

\noindent $R_1)$. $\alpha'(t)e^{\gamma t},\alpha^{''}(t)e^{\gamma t},\beta'(t)e^{\gamma t},\beta^{''}(t)e^{\gamma t}$
 are bounded and uniformly continuous functions. \\
 
\noindent $R_2)$. Let $p,q >1$ such that $1/p+1/q=1$ and for all $ t\geq0$, $$\max\{||F_{21}(t)||,||F_{22}(t)||\}\leq \frac{\gamma e^{-\gamma t}}{pM}, \ \
\max\{||F_{21}'(t)||,||F_{22}'(t)||\}\leq \frac{\gamma^2 e^{-\gamma t}}{(pM)^2}\, .$$
R.C. Grimmer has shown that the equation (\ref{EQ3}) has a resolvent operator $\{R(t)\}_{t\geq 0}$ which satisfies condition {\bf (R)} ; that is,  $\forall t \geq0: \
||R(t)||\leq Me^{-\frac{\gamma t}{q}}$ (see \cite{RCGrim} for original references).

Now, let us consider the following conditions

\begin{enumerate}
\item[(C1)]  $y_0(t)=R(t)(u_0+h(0))+\int_0^tR(t-s)f(s,0)ds.$
\item[(C2)]  $h(u)=\begin{bmatrix}
  {h_1(\xi)}\\{h_2(\eta)}
 \end{bmatrix} $, where $h_1: C(\mathbb{R}^+; H_0^1(\Omega))\to H_0^1(\Omega)$ and $h_2: C(\mathbb{R}^+; L^2(\Omega))\to L^2(\Omega)$ are Lipschitz with the same Lipschitz's constant $\dfrac{\rho}{pM(\rho+||y_0||_{\infty})}$.
 \item[(C3)]   $b: H_0^1(\Omega)\to L^2(\Omega)$ is also Lipschitz, with Lipschitz's constant $$\dfrac{\gamma \rho}{qM||a||(\rho +||y_0||_{\infty})}.$$
\end{enumerate} 

Finally, we have the following theorem as an application to the heat conduction in materials with memory
\begin{thm} Let $a\in \mathrm{R} AAA(\mathbb{R}^+;\mathbb{R})$ 
and consider a real number $\rho >0$ such that:
\begin{equation}\label{EQ4}
 \rho \geq M \left( ||u_0||+ ||h(0)||+\dfrac{q}{\gamma}||a||\;||b(0)||\right).
\end{equation}
Suppose that the conditions $(C1)-(C3)$ holds. Then, equation (\ref{EQ3})-(\ref{EQ33}) has a unique mild solution 
$y\in \mathrm{R}AAA(\mathbb{R}^{+}; X)$ such that $||y-y_0||_{\infty}\leq \rho$.
\end{thm}

\end{document}